\documentclass[12pt, reqno]{amsart}
\usepackage{amssymb,amsthm,amsmath}
\usepackage[numbers,sort&compress]{natbib}
\usepackage{amssymb,amsmath}
\usepackage{amsfonts}
\usepackage{mathrsfs}
\usepackage{latexsym}
\usepackage{amssymb}
\usepackage{amsthm}
\usepackage{indentfirst}
\date{\today}

\let\oldsection\section
\renewcommand\section{\setcounter{equation}{0}\oldsection}

\newtheorem{corollary}{Corollary}[section]
\newtheorem{theorem}{Theorem}[section]

\newtheorem{proposition}{Proposition}[section]
\newtheorem{definition}{Definition}[section]

\newtheorem{remark}{Remark}[section]

\allowdisplaybreaks
\begin{document}

\title[Entropy-bounded solutions 1D compressible Navier-Stokes]{Entropy-bounded solutions to the compressible Navier-Stokes equations: with far field vacuum}

\author{Jinkai~Li}
\address[Jinkai~Li]{Department of Mathematics, The Chinese University of Hong Kong, Hong Kong, P.R.China}
\email{jklimath@gmail.com}

\author{Zhouping~Xin}
\address[Zhouping~Xin]{The Institute of Mathematical Sciences,
The Chinese University of Hong Kong, Hong Kong, P.R.China}
\email{zpxin@ims.cuhk.edu.hk}


%%% ----------------------------------------------------------------------

\begin{abstract}
The entropy is one of the fundamental states of a fluid and, in the
viscous case, the equation that it satisfies is highly singular in the
region close to the vacuum. In spite of its importance in the gas
dynamics, the mathematical analyses on the behavior of the entropy
near the vacuum region, were rarely carried out; in particular, in the
presence of vacuum, either at the far field or
at some isolated interior points, it was unknown if the
entropy remains its boundedness. The results
obtained in this paper indicate that the ideal
gases retain their uniform
boundedness of the entropy, locally or globally in time,
if the vacuum occurs at the far field only and the density decays
slowly enough at the far field. Precisely, we consider the Cauchy
problem to the one-dimensional full compressible Navier-Stokes
equations without heat conduction, and establish the local and global
existence and uniqueness of entropy-bounded solutions, in the presence
of vacuum at the far field only. It is also shown that, different from
the case that with compactly supported initial density,
the compressible Navier-Stokes equations, with slowly decaying
initial density, can propagate the regularities
in the inhomogeneous Sobolev spaces.
\end{abstract}

%%% ----------------------------------------------------------------------
\maketitle

%\tableofcontents
\allowdisplaybreaks

\section{Introduction}
Let $\rho, u,$ and $\theta$ be the density, velocity, and temperature
of a fluid, and denote $t$ and
$x$ as the time and spatial variables.
Then, the full compressible Navier-Stokes equations read as
\begin{eqnarray}
&&\partial_t\rho+\text{div}\,(\rho u)=0,\label{GCNSrho}\\
&&\partial_t(\rho u)+\text{div}\,(\rho u\otimes u)=\text{div}
  \,\mathbb T+\rho f,\label{GCNSu}\\
&&\partial_t(\rho E)+\text{div}\,(\rho
uE)+\text{div}\,q=\text{div}\,(\mathbb Tu)+\rho\mathcal Q\label{GCNSE},
\end{eqnarray}
where $E=\frac{|u|^2}{2}+e$ is the specific total energy,
$e=e(\rho,\theta)$ is the specific internal energy, $\mathbb T$
is the stress tensor, $q$ is the internal
energy flux directly related to the transfer
of heat, $f$ is the external force, and $\mathcal Q$ is the external
heat source.

By (\ref{GCNSrho})--(\ref{GCNSE}), one
can obtain the following equation for $e$:
\begin{equation}
  \partial_t(\rho e)+\text{div}\,(\rho ue)+p\text{div}\,u+\text{div}\,q
  =\mathbb S:\nabla u+\rho\mathcal Q. \label{GCNSe}
\end{equation}
The stress tensor $\mathbb T$
is given by
\begin{eqnarray*}
\mathbb T=\mathbb S-pI, \quad \mathbb S=2\mu\mathbb
Du+\lambda\text{div}\,uI,\quad \mathbb Du=\frac12(\nabla u+(\nabla
u)^T),
\end{eqnarray*}
where $I$ is the $3\times3$ identity matrix, $p$ is the pressure, and
$\mu$ and $\lambda$ are viscosity coefficients, satisfying $\mu>0$
and $2\mu+3\lambda\geq0$. In this paper, we consider the ideal gases,
and state equations are
$$
p=R\rho\theta, \quad e=c_v\theta,
$$
for two positive constants $R$ and $c_v$. Then, it follows
from (\ref{GCNSe}) that
\begin{equation}
c_v[\partial_t(\rho\theta)+\text{div}\,(\rho
u\theta)]+p\text{div}\,u+\text{div}\,q=\mathbb S:\nabla u+\rho\mathcal
Q. \label{GCNStheta}
\end{equation}

Recalling the state equations for $p$ and $e$, by the Gibb's equation
$\theta Ds=De+pD(\frac1\rho)$, where $s$ is the state of the 
entropy, one
has the following relationship between $p$ and $s$:
$$
p=Ae^{\frac{s}{c_v}}\rho^\gamma,
$$
for some positive constant $A$, where $\gamma-1=\frac{R}{c_v}$. Thanks
to this, and using the state equations for $p$ and $e$ again,
one can derive from (\ref{GCNSrho}), (\ref{GCNSu}), and
(\ref{GCNStheta}) the following
equation for the entropy $s$:
\begin{equation*}
\partial_t(\rho s)+\text{div}\,(\rho s u)+\text{div}\,\left(\frac
q\theta\right)=\frac1\theta\left(\mathbb S:\nabla
u-\frac{q\cdot\nabla\theta}{\theta}\right)+\rho\frac{\mathcal
Q}{\theta}.
\end{equation*}
For the internal energy flux $q$,
by the Fourier's law of heat conduction, we assume that
$q=-\kappa\nabla\theta,$
where $\kappa\geq0$ is the heat conduction coefficient.

There are extensive literatures on the mathematical analyses of the
compressible Navier-Stokes equations. In the absence of vacuum, that is
the density is bounded from below by some positive constant, the local
well-posedness results
were proved by Nash \cite{NASH62}, Itaya \cite{ITAYA71},
Vol'pert--Hudjaev \cite{VOLHUD72}, Tani \cite{TANI77}, Valli
\cite{VALLI82}, and Lukaszewicz \cite{LUKAS84}.
The first global well-posedness result was established by
Kazhikhov--Shelukhin \cite{KAZHIKOV77}, where they proved the
global well-posedness of strong solutions of the initial boundary value
problem to the one-dimensional compressible
Navier-Stokes equations, for arbitrary $H^1$ initial data, and the
corresponding result for the Cauchy problem was later
proved by Kazhikhov \cite{KAZHIKOV82}; global well-posedness of weak
solutions to the one-dimensional
compressible Navier-Stokes equations was proved by
Zlotnik--Amosov \cite{ZLOAMO97,ZLOAMO98} and by Chen--Hoff--Trivisa
\cite{CHEHOFTRI00} for the initial boundary
value problems, and by Jiang--Zlotnik \cite{JIAZLO04} for the Cauchy
problem. Large time behavior of solutions to the one dimensional
compressible Navier-Stokes equations with large initial data was
recently proved by Li--Liang \cite{LILIANG16}.
For the multi-dimensional case, the global well-posedness
of strong solutions were established only for
small perturbed initial data around some non-vacuum equilibrium or for
spherically symmetric large initial data, see Matsumura--Nishida
\cite{MATNIS80,MATNIS81,MATNIS82,MATNIS83}, Ponce \cite{PONCE85},
Valli--Zajaczkowski \cite{VALZAJ86}, Deckelnick \cite{DECK92}, Jiang
\cite{JIANG96}, Hoff \cite{HOFF97}, Kobayashi--Shibata \cite{KOBSHI99},
Danchin \cite{DANCHI01}, and Chikami--Danchin \cite{CHIDAN15}.
One of the major differences between one dimensional case from the
multi-dimensional one is that if no vacuum is contained initially, then
no vacuum will form later on in finite time, for the one dimensional
compressible Navier-Stokes equations, as shown by Hoff-Smoller
\cite{HS}, while the similar result remains unknown 
for the multi-dimensional case.

In the presence of vacuum, that is the density may vanish
on some set, or tends to zero at the far field, the breakthrough
was made by Lions \cite{LIONS93,LIONS98}, where he proved
the global existence of weak solutions to the isentropic compressible
Navier-Stokes equations, with adiabatic constant $\gamma\geq\frac95$;
the requirement on $\gamma$ was later relaxed by
Feireisl--Novotn\'y--Petzeltov\'a \cite{FEIREISL01} to $\gamma>\frac32$,
and further by Jiang--Zhang \cite{JIAZHA03} to $\gamma>1$ but only
for the axisymmetric solutions. For the full compressible Navier-Stokes
equations, global existence of the variational weak solutions was proved
by Feireisl \cite{FEIREISL04P,FEIREISL04B}; however, due to the
assumptions on the constitutive equations made in
\cite{FEIREISL04P,FEIREISL04B}, the ideal gases were not included there.
Local well-posedness of strong solutions, in the presence of vacuum,
was proved first
for the isentropic case by Salvi--Stra$\check{\text s}$kraba
\cite{SALSTR93}, Cho--Choe--Kim \cite{CHOKIM04}, and Cho--Kim
\cite{CHOKIM06-1}, and later for the polytropic case by Cho--Kim
\cite{CHOKIM06-2}. It should be noticed that, in
\cite{SALSTR93,CHOKIM04,CHOKIM06-1,CHOKIM06-2}, the
solutions were established in the homogeneous Sobolev spaces, that is,
it is $\sqrt\rho u$ rather than $u$ itself that has the $L^\infty(0,T;
L^2)$ regularity. Generally, one can not expect that the strong
solutions to the compressible Navier-Stokes equations lie in
the inhomogeneous Sobolev spaces, if the initial density has compact
support. Actually, it was proved recently by Li--Wang--Xin
\cite{LWX} that: neither isentropic nor the full compressible
Navier-Stokes equations on $\mathbb R$,
with $\kappa=0$ for the full case,
has any solution $(\rho, u,\theta)$ in the inhomogeneous
Sobolev spaces $C^1([0,T];
H^m(\mathbb R))$, with $m>2$, if $\rho_0$ is compactly supported and
some appropriate conditions on the initial data
are satisfied; the $N$-dimensional full
compressible Navier-Stokes equations, with positive heat
conduction, have no solution $(\rho, u,\theta)$, with finite entropy,
in the inhomogeneous Sobolev spaces
$C^1([0,T]; H^m(\mathbb R^N))$, with $m>[\frac N2]+2$,
if $\rho_0$ is compactly supported.
Global existence of strong and classical solutions
to the compressible Navier-Stokes equations, in the presence of
initial vacuum, was first proved by Huang--Li--Xin \cite{HLX12},
where they established the global well-posedness of
strong and classical solutions, with small initial
basic energy, to the three-dimensional isentropic compressible
Navier-Stokes equations, see Li--Xin
\cite{LIXIN13} for further developments. However, due to the finite
in time blow-up results by Xin \cite{XIN98} and Xin--Yan
\cite{XINYAN13}, one can not expect the global well-posedness of
classical solutions, in either inhomogeneous or homogeneous Sobolev
spaces, to the full compressible Navier-Stokes equations in the presence
of vacuum. In particular, it was proved in
\cite{XINYAN13} that, for the full compressible Navier-Stokes equations,
if initially there is an isolated
mass group surrounded by the vacuum region, then for the
case $\kappa=0$, any classical solution
must blow-up in finite time, and for the case $\kappa
>0$, any classical
solutions, with finite entropy in the vacuum region, must blow-up
in finite time. Global existence of strong solutions
to the heat conducting
full compressible Navier-Stokes equations were obtained by
Huang--Li \cite{HUANGLI11} for the case that with non-vacuum far
field, and by Wen--Zhu \cite{WENZHU17} for the case that with vacuum
far field. The spaces of the solutions obtained in
\cite{HUANGLI11,WENZHU17} can not exclude the possibility
that the entropy is infinite somewhere
in the vacuum region, even if it is
initially finite; in fact, due to the results in \cite{XINYAN13},
the corresponding entropy in \cite{HUANGLI11,WENZHU17} must be
infinite somewhere in the vacuum region, if initially
there is an isolated
mass group surrounded by the vacuum region.

Recalling the state equations for the ideal gases, the entropy can be
expressed in terms of the density and temperature as
$$
s=c_v\left(\log\frac RA+\log \theta-(\gamma-1)\log\rho\right),
$$
from which one can see that the entropy may develop singularities or
even is not well defined in the vacuum region and, consequently,
it is impossible to obtain the desired regularities of $s$ merely
from those of $\theta$ and $\rho$, in the presence of vacuum. Therefore,
though the vacuums are allowed for the solutions established
in \cite{CHOKIM06-2,HUANGLI11,WENZHU17} by choosing $(\rho, u, \theta)$
as the unknowns, no regularities of the entropy $s$ can be implied in
the vacuum region there and, due to the result in \cite{XINYAN13},
the entropy of the solutions obtained in \cite{HUANGLI11,WENZHU17}
must be infinite in the vacuum region.
To the best of our knowledge, in the existing
literatures, there were no such results that provided the uniform lower
or upper bounds of the entropy near the vacuum.

As stated in the previous paragraph, since
the entropy can not be even defined at the places where the density
vanishes, it may be unreasonable to study the entropy
for the full compressible Navier-Stokes equations of the ideal gases,
if the vacuum region is an open set; however, when the vacuum occurs
only at some isolated interior points or at the far field and if,
moreover, the entropy behaves well when the fluid tends to these vacuum
points or to the far field, it is still possible to define the entropy
there. Therefore, a natural question is what kind of behavior of the
entropy, at the vacuum far field or near the isolated interior vacuum
points, can be preserved by the ideal gases, when the flow evolves.
The aim of this paper is to give some answers to this question and, in
particular, as indicated in our main results, the ideal gases
can preserve their boundedness of the entropy, locally or globally in
time, if the vacuum happens at the far field only.

Another question that we want to address in this paper is: under what
kind of assumptions on the initial density, beyond the the case that
the initial density is uniformly away from the vacuum, the compressible
Navier-Stokes equations admit solutions in the inhomogeneous Sobolev
spaces. On one hand,
recalling the result in \cite{LWX}, for the case that the initial
density has a compact support, the compressible Navier-Stokes
equations are ill-posed in the inhomogeneous Sobolev spaces; on the other hand, for the case that the initial
density is uniformly away from the vacuum, the
compressible Navier-Stokes equations are well-posed in the
inhomogeneous Sobolev spaces. Comparing these two cases, by
understanding the case that with compact support as having
supper fast decay at the far field, it is natural to ask \textit{if
the fast decay of the density can cause the ill-posedness of the
compressible Navier-Stokes equations in the inhomogeneous spaces,
or if the compressible Navier-Stokes equations will be well-posed
in the inhomogeneous spaces when the initial density
decays slowly at the far field}. We will show in this paper that if
the initial density decays slower than $\frac{K_0}{|x|^2}$, for some
positive constant $K_0$, at the far field, then
the compressible Navier-Stokes equations are indeed well-posed in the
inhomogeneous Sobolev spaces, where $K_0$ is an arbitrary positive
constant. Note that this is consistent with the well-posedness result
for the compressible Navier-Stokes equations in the inhomogeneous
Sobolev spaces in the absence of vacuum.

In this paper, we consider the one dimensional case, and assume that there are no external forces and
heating source, i.e.\,$f\equiv\mathcal Q\equiv0$, and that
there is no heat conduction in the fluids, that is $\kappa=0$,
while the muti-dimensional case and the cases that with heat conduction
will be studied in the further works.
Under these assumptions, the system considered in this paper is the
following one-dimensional compressible Navier-Stokes equations:
\begin{eqnarray}
\rho_t+(\rho u)_x&=&0,\label{ECNSrho}\\
\rho(u_t+uu_x)-\mu u_{xx}+p_x&=&0,\label{ECNSu}\\
c_v[(\rho\theta)_t+(\rho u\theta)_x]+pu_x&=&\mu(u_x)^2.\label{ECNStheta}
\end{eqnarray}

Recalling the state equation $p=R\rho\theta$ and
$c_v=\frac{R}{\gamma-1}$, equation (\ref{ECNStheta}) can be rewritten
equivalently as a equation for the pressure $p$, that is
\begin{equation}\label{ECNSp}
  p_t+up_x+\gamma u_xp=\mu(\gamma-1)(u_x)^2.
\end{equation}
As will be seen later, it is more convenient to use
(\ref{ECNSp}), instead of (\ref{ECNStheta}), to state and prove
the results, in other words, we will use the pressure, instead of the
temperature, as one of the unknowns, throughout this paper; however,
it should be mentioned that, as we consider the case that the vacuum happens only at the far field, (\ref{ECNSp}) is equivalent to (\ref{ECNStheta}).

We will consider the Cauchy problem and, therefore, complement system system (\ref{ECNSrho}), (\ref{ECNSu}), and (\ref{ECNSp}), with the following initial condition
\begin{equation}
  \label{ECNSic}
  (\rho, u, p)|_{t=0}=(\rho_0, u_0, p_0).
\end{equation}

Before stating the main results, we first clarify some necessary notations
being used throughout this paper.
For $1\leq q\leq\infty$ and positive integer $m$, we use $L^q=L^q(\mathbb R)$ and $W^{1,q}=W^{m,q}(\mathbb R)$ to denote the standard Lebesgue and Sobolev spaces, respectively, and in the case that $q=2$, we use $H^m$
instead of $W^{m,2}$. For simplicity, we also use the notations $L^q$ and
$H^m$ to denote the $N$ product spaces $(L^q)^N$ and $(H^m)^N$, respectively.
We always use $\|u\|_q$ to denote the $L^q$ norm of $u$. For shortening the
expressions, we sometimes use $\|(f_1,f_2,\cdots,f_n)\|_X$ to denote the sum
$\sum_{i=1}^N\|f_i\|_X$ or its equivalent norm $\left(\sum_{i=1}^N\|f_i\|_X^2
\right)^{\frac12}$.

We have the following two theorems on the local and global well-posedness of solutions to system (\ref{ECNSrho}), (\ref{ECNSu}), and (\ref{ECNSp}), subject to (\ref{ECNSic}):

\begin{theorem}[\textbf{Local well-posedness}]
\label{localEuler}
Assume that
\begin{equation}
  \inf_{y\in(-R,R)}\rho_0(y)>0,\quad\forall R\in(0,\infty),\quad\left(\rho_0',\sqrt{\rho_0}u_0, u_0',p_0,\frac{p_0'}{\sqrt{\rho_0}}\right)\in L^2,\label{ASM1}
\end{equation}
and denote $F_0:=\mu u_0'-p_0$.

Then, the following two hold:

(i) There is a unique local solution $(\rho, u, p)$ to system (\ref{ECNSrho})--(\ref{ECNSu}) and (\ref{ECNSp}), subject to (\ref{ECNSic}), satisfying
\begin{eqnarray*}
&\rho-\rho_0\in C([0,T];L^2),\quad\rho_x\in L^\infty(0,T; L^2),\quad\rho_t \in L^\infty(0,T; L^2((-r,r))), \\
&\sqrt{\rho}u\in C([0,T]; L^2),\quad u_x\in L^\infty(0,T; L^2)\cap L^2(0,T; H^1),\quad  \sqrt{\rho}u_t\in L^2(0,T; L^2),\\
&p-p_0\in C([0,T]; L^2),\quad p_x\in L^\infty(0,T; L^2),\quad p_t\in L^4(0,T; L^2((-r,r))),
\end{eqnarray*}
for all $r\in(0,\infty)$,where $T$ is a positive constant depending only on $\gamma, \mu, \|\rho_0\|_\infty, \|v_0'\|_2,$ $\|p_0\|_2$, and $\|p_0\|_\infty$.

(ii) Assume in addition that
\begin{equation}\label{ASM2}
\left|\left(\frac{1}{\sqrt{\rho_0}}\right)'(y)\right|\leq \frac{K_0}{2},\quad \forall y\in \mathbb R, \quad\rho_0^{-\frac\delta2}F_0\in L^2,
\end{equation}
for two positive constants $\delta$ and $K_0$. Then, $(\rho,u,p)$ has the additional regularities
\begin{equation}\label{CON1}
  \left\{
  \begin{aligned}
    &u\in L^\infty(0,T; H^1),&&\mbox{if }u_0\in H^1
    \mbox{ and }\delta\geq1,\\
    &\theta \in L^\infty(0,T; H^1),&&\mbox{if }\theta_0\in H^1\mbox{ and }\delta\geq1,\\
    &s\in L^\infty(0,T; L^\infty),&&\mbox{if }s_0\in L^\infty\mbox{ and }\delta\geq\gamma,
  \end{aligned}
  \right.
\end{equation}
where $\theta:=\frac{p}{R\rho}$ and $s:=c_v\log\left(\frac{p}{A\rho^\gamma}\right)$, respectively, are the temperature and entropy.
\end{theorem}

\begin{theorem}[\textbf{Global well-posedness}]
\label{globalEuler}
Assume that (\ref{ASM1}) holds, and that
both $\rho_0$ and $p_0$ are in $L^1.$ Then, there is a unique global
solution $(\rho, u, p)$ to system (\ref{ECNSrho})--(\ref{ECNSu}) and
(\ref{ECNSp}), subject to (\ref{ECNSic}), satisfying the regularities
stated in (i) of Theorem \ref{localEuler}, for any $T\in(0,\infty)$.
Moreover, if assume in addition that (\ref{ASM2}) holds, then
(\ref{CON1}) holds for any $T\in(0,\infty)$.
\end{theorem}

Theorem \ref{localEuler} and Theorem \ref{globalEuler}, respectively, are the corollaries of the more accurate results Theorem \ref{LOCAL} and Theorem \ref{GLOBAL}, being stated in the next
section in the Lagrangian coordinates. Therefore, in the rest of this paper, we focus on studying the compressible Navier-Stokes equations
in the Lagrangian coordinates.

The rest of this paper is arranged as follows: in Section \ref{secresult}, we reformulate the system in the Lagrangian
coordinates, and state our main results; in Section \ref{seclocnov},
we consider the system in the absence of
vacuum, and carry out some a priori estimates, which are
independent of the positive lower bound of the density;
the proof of Theorem \ref{LOCAL} is given in Section \ref{seclocv},
while that of Theorem \ref{GLOBAL} is given in the last section.

\section{Reformulation in Lagrangian coordinates and main results}
\label{secresult}
Let $y$ be the Lagrangian coordinate, and define the coordinate transform
between the Lagrangian coordinate $y$ and the Euler coordinate $x$ as
\begin{equation*}
  x=\eta(y,t),
\end{equation*}
where
$\eta(y,t)$ is the flow map determined by $u$, that is,
\begin{equation*}\label{flowmap}
  \left\{
  \begin{array}{l}
  \partial_t\eta(y,t)=u(\eta(y,t),t),\\
  \eta(y,0)=y.
  \end{array}
  \right.
\end{equation*}

Denote by $\varrho, v,$ and $\pi$ the density, velocity, and pressure, respectively, in the Lagrangian coordinate, that is we define
\begin{equation*}
  \varrho(y,t):=\rho(\eta(y,t),t),\quad v(y,t):=u(\eta(y,t),t), \quad \pi(y,t):=p(\eta(y,t),t). \label{newunknown}
\end{equation*}
Recalling the definition of $\eta(y,t)$, by straightforward calculations, one can check that
\begin{eqnarray*}
  (u_x,p_x)=\left(\frac{v_y}{\eta_y},\frac{\pi_y}{\eta_y}\right), \quad u_{xx}=\frac{1}{\eta_y}\left(\frac{v_y}
  {\eta_y}\right)_y,\\
  \rho_t+u\rho_x=\varrho_t,\quad u_t+uu_x=v_t,\quad p_t+up_x=\pi_t.
\end{eqnarray*}
Define the function $J=J(y,t)$ as
\begin{equation*}
  J(y,t)=\eta_y(y,t),
\end{equation*}
then it follows that 
$$
J_t=v_y.
$$

Thanks to the above, system (\ref{ECNSrho}), (\ref{ECNSu}), and (\ref{ECNSp}) can be rewritten in the Lagrangian coordinate as
\begin{eqnarray}
J_t&=&v_y\label{LCNSJ}\\
\varrho_t+\frac{v_y}{J}\varrho&=&0,\label{LCNSvarrho}\\
\varrho v_t-\frac{\mu}{J}\left(\frac{v_y}{J}\right)_y+\frac{\pi_y}{J}&=&0,\label{LCNSv}\\
\pi_t+\gamma\frac{v_y}{J}\pi&=&(\gamma-1)\mu\left(\frac{v_y}{J}\right)^2. \label{LCNSpi}
\end{eqnarray}

One can further reduce the above system to a simpler version. In fact,
due to (\ref{LCNSJ}) and (\ref{LCNSvarrho}), it holds that
$$
(J\varrho)_t=J_t\varrho+J\varrho_t=v_y\varrho-J\frac{v_y}{J}\varrho=0,
$$
from which, by setting $\varrho|_{t=0}=\varrho_0$ and noticing that $J|_{t=0}=1$, we have
\begin{equation*}
  J\varrho=\varrho_0.
\end{equation*}
Therefore, one can drop (\ref{LCNSvarrho}) from system (\ref{LCNSJ})--(\ref{LCNSpi}), and rewrite (\ref{LCNSv}) as
\begin{equation*}
  \varrho_0v_t-\mu\left(\frac{v_y}{J}\right)_y+ \pi_y =0.
\end{equation*}

In summary, we only need to consider the following system
\begin{eqnarray}
J_t&=&v_y,\label{lcnsJ}\\
\varrho_0v_t-\mu\left(\frac{v_y}{J}\right)_y+\pi_y&=&0,\label{lcnsv}\\
\pi_t+\gamma \frac{v_y}{J}\pi&=&(\gamma-1)\mu\left(\frac{v_y}{J}\right)^2.\label{lcnspi}
\end{eqnarray}
As will be shown later, the effective viscous flux $G$ defined as
\begin{equation*}
  G:=\mu\frac{v_y}{J}-\pi
\end{equation*}
plays a crucial role in proving the global existence of solutions to system (\ref{lcnsJ})--(\ref{lcnspi}). By straightforward calculations, it follows from (\ref{lcnsv}) and (\ref{lcnspi}) that
\begin{equation}\label{eqG}
  G_t-\frac{\mu}{J}\left(\frac{G_y}{\varrho_0}\right)_y=-\gamma\frac{v_y}{J}G.
\end{equation}

We will consider the Cauchy problem and, thus, complement system (\ref{lcnsJ})--(\ref{lcnspi}) with the initial condition
\begin{equation}
  \label{IC}
  (J,v,\pi)|_{t=0}=(J_0,v_0,\pi_0),
\end{equation}
where $J_0$ has uniform positive lower and upper bounds.

It should
be pointed out that, by the definition of $J$, the initial $J_0$ should
be identically one; however, for the aim of extending a local
solution $(J,v,\pi)$ to be a global one, we need the local existence
of solutions to system (\ref{lcnsJ})--(\ref{lcnspi}), with initial $J_0$
not being identically one. Therefore, in this paper, when studying the
local solutions, the initial $J_0$ is allowed to be not identically one,
but when studying the global solutions, we always assume that $J_0$ is
identically one.

Local and global strong solution to system (\ref{lcnsJ})--(\ref{lcnspi}), subject to (\ref{IC}), are defined in
the following two definitions.

\begin{definition}
\label{defloc}
Given a positive time $T\in(0,\infty)$.
A triple $(J,v,\pi)$ is called a strong solution to system (\ref{lcnsJ})--(\ref{lcnspi}), subject to (\ref{IC}), on $\mathbb R\times(0,T)$, if
it has the properties
\begin{eqnarray*}
&\inf_{y\in\mathbb R,t\in(0,T)}J(y,t)>0,\quad\pi\geq0\mbox{ on }\mathbb R\times(0,T),
\\
&J-J_0\in C([0,T];L^2),\quad\frac{J_y}{\sqrt{\varrho_0}}\in L^\infty(0,T; L^2),\quad J_t\in L^\infty(0,T; L^2),\\
&
\sqrt{\varrho_0}v\in C([0,T]; L^2),\quad v_y\in L^\infty(0,T; L^2),\quad\left(\sqrt{\varrho_0}v_t,\frac{v_{yy}}{\sqrt{\varrho_0}}\right)\in L^2(0,T; L^2),\\
& \pi\in C([0,T]; L^2),\quad \frac{\pi_y}{\sqrt{\varrho_0}}\in L^\infty(0,T; L^2),\quad\pi_t\in L^4(0,T; L^2),
\end{eqnarray*}
satisfies equations (\ref{lcnsJ})--(\ref{lcnspi}), a.e.\,in $\mathbb R\times(0,T)$, and fulfills the initial condition (\ref{IC}).
\end{definition}

\begin{definition}
\label{defglo}
A triple $(J,v,\pi)$ is called a global strong solution to system (\ref{lcnsJ})--(\ref{lcnspi}), subject to (\ref{IC}), if it is a
strong solution to the same system on $\mathbb R\times(0,T)$, for any
positive time $T\in(0,\infty)$.
\end{definition}

The main results of this paper are the following two theorems cocerning the local and global existence of strong solutions to system (\ref{lcnsJ})--(\ref{lcnspi}), subject to (\ref{IC}).

\begin{theorem}[\textbf{Local well-posedness}]
\label{LOCAL}
Assume that
\begin{gather}
  \inf_{y\in(-R,R)}\varrho_0(y)>0,~~\forall R\in(0,\infty),\quad \varrho_0\leq\bar\varrho\mbox{ on }\mathbb R,
  \tag{H1}\\
  \left(\sqrt{\varrho_0}v_0, v_0',\pi_0,\frac{\pi_0'}{\sqrt{\varrho_0}}\right)\in L^2,\quad\pi_0\geq0\mbox{ on }\mathbb R,\tag{H2}\\
  0<\underline J\leq J_0\leq\bar J<\infty\mbox{ on }R,\quad \frac{J_0'}{\sqrt{\varrho_0}}\in L^2,\nonumber
\end{gather}
for positive constants $\bar\varrho, \underline J,$ and $\bar J$,
and denote $G_0:=\mu v_0'-\pi_0.$

The following two hold:

(i) There is a positive time $T$ depending only on $\gamma, \mu, \bar\varrho, \underline J, \bar J, \|v_0'\|_2, \|\pi_0\|_2$, and $\|\pi_0\|_\infty$, such that system (\ref{lcnsJ})--(\ref{lcnspi}), subject to the initial condition (\ref{IC}),
has a unique strong solution $(J,v,\pi)$, on $\mathbb R\times(0,T)$.

(ii) Assume in addition that
\begin{equation*}
\left|\left(\frac{1}{\sqrt{\varrho_0}}\right)'(y)\right|\leq \frac{K_0}{2},\quad \forall y\in \mathbb R, \quad\varrho_0^{-\frac\delta2}G_0\in L^2,\tag{H3}
\end{equation*}
for two positive constants $\delta$ and $K_0$.

Then, $(J,v,\pi)$ has the additional regularities
\begin{eqnarray}
&\varrho_0^{-\frac\delta2}G\in L^\infty(0,T; L^2)\cap L^4(0,T;L^\infty),\quad
\varrho_0^{-\frac{\delta+1}{2}}G_y\in L^2(0,T; L^2),\label{ADR1}
\end{eqnarray}
where $G:=\mu\frac{v_y}{J}-\pi$ is the effective viscous flux,
and
\begin{equation}\label{ADR2}
  \left\{
  \begin{aligned}
    &v\in L^\infty(0,T; H^1),&&\mbox{if }v_0\in H^1
    \mbox{ and }\delta\geq1,\\
    &\vartheta \in L^\infty(0,T; H^1),&&\mbox{if }\vartheta_0\in H^1, \frac{J_0'}{\varrho_0}\in L^2,\,\mbox{and }\delta\geq1,\\
    &s\in L^\infty(0,T; L^\infty),&&\mbox{if }s_0\in L^\infty\mbox{ and }\delta\geq\gamma,
  \end{aligned}
  \right.
\end{equation}
where $\vartheta:=\frac{\pi}{R\varrho}$ and $s:=c_v\log\left(\frac{\pi}{A\varrho^\gamma}\right)$, respectively, are the corresponding temperature and entropy,
with $\varrho:=\frac{\varrho_0}{J}$ being the density, which satisfies
equation (\ref{LCNSvarrho}), and $\vartheta_0:=\frac{\pi_0}{R\varrho_0}$
and $s_0:=c_v\log\left(\frac{\pi_0}{A\varrho_0^\gamma}\right)$, respectively
are the initial temperature and entropy.
\end{theorem}

\begin{remark}
\label{remarkrhocondition}
Basically, the condition $\big|\big(\tfrac{1}{\sqrt{\varrho_0}}\big)'\big|\leq \frac{K_0}{2}$ or equivalently $|\varrho_0'|\leq K_0\varrho_0^{\frac32}$ on
$\mathbb R$ means that $\varrho_0$ decays no faster than $\frac{K}{y^2}$
at the far field: if choosing
$$
\varrho_0(y)=\frac{K_\varrho}{\langle y\rangle^{\ell_\varrho}},\quad0<K_\varrho<\infty, 0\leq\ell_\varrho<\infty,\quad\mbox{where }\langle y\rangle =(1+y^2)^{\frac12},
$$
then
$$
\left|\left(\frac{1}{\sqrt{\varrho_0}}\right)'\right|\leq \frac{K_0}{2}\mbox{ on }\mathbb R\quad \Longleftrightarrow \quad\ell_\varrho\leq2.
$$
\end{remark}

\begin{remark}
\label{remarkinitialdataloc}
Choose
$$
\varrho_0(y)=\frac{K_\varrho}{\langle y\rangle^{\ell_\varrho}},\quad J_0\equiv1,\quad v_0\in C_c^\infty,\quad
\pi_0=Ae^{\frac{1}{c_v}}\varrho_0^\gamma,
$$
where $K_\varrho$ and $\ell_\varrho$ are positive numbers.

(i) If $\frac{1}{2\gamma-1}<\ell_\varrho\leq2$, then
$(\varrho_0,v_0,\pi_0)$ satisfies conditions (H1), (H2), and (H3), with
$\delta=1$. Therefore, by Theorem \ref{LOCAL}, there
is a unique local strong solution $(J,v,\pi)$, with $v$ being in
the inhomogeneous Sobolev space $L^\infty(0,T; H^1)$; if moreover that
$\gamma>\frac54$ and $\frac{1}{2(\gamma-1)}<\ell_\varrho\leq2$, then
$\vartheta_0\in H^1$, and, consequently, the temperature $\vartheta$
also lies in the inhomogeneous Sobolev space $L^\infty(0,T; H^1)$.
Note that this does not contradict to the ill-posedness results for
the compressible Navier-Stokes equations in \cite{LWX}, as the initial
density there was assumed to be compactly supported.

(ii) If $\frac1\gamma<\ell_\varrho\leq2$, then $(\varrho_0,v_0,\pi_0)$
satisfies conditions (H1), (H2), (H3), with $\delta=\gamma$, and
$s_0\equiv1$. Therefore, by Theorem \ref{LOCAL}, there
is a unique local strong solution $(J,v,\pi)$,
and the corresponding entropy $s$ is uniformly bounded on $\mathbb
R\times(0,T)$. To the best of our knowledge, this is the first
time that the boundedness of the entropy is achieved, in the presence of
vacuum at the far field, for the compressible Navier-Stokes equations.

(iii) Combining (i) with (ii), if $\gamma>\frac54$ and
$\max\left\{\frac1\gamma,\frac{1}{2(\gamma-1)}\right\}<\ell_\varrho\leq2$,
there is a unique local strong solution $(J,v,\pi)$, with the properties that
the corresponding entropy is uniformly bounded, and the velocity and the
corresponding temperature lie in the inhomogeneous space $L^\infty(0,T;
H^1)$.
\end{remark}

\begin{remark}
  \label{remarkpropogation}
  (i) The assumption $\inf_{y\in(-R,R)}\varrho_0(y)>0,$ for all
  $R\in(0,\infty)$, is used for the boundedness of the entropy and
  the regularities of the velocity and temperature in the inhomogeneous
  Sobolev spaces, but it is not needed for the local well-posedness (in the
  homogeneous Sobolev spaces).

  (ii) The compressible Navier-Stokes equations propagate the regularities in the homogeneous Sobolev spaces, see \cite{CHOKIM04,CHOKIM06-1,CHOKIM06-2}, but not in the inhomogeneous Sobolev spaces (in particular, the $L^2$ regularity of $v$
  can not be propagated), see \cite{LWX}, if the initial density has a compact support.
  While (ii) of Theorem \ref{LOCAL} shows that, if the initial density
  decays slowly to the vacuum at the far field, then the regularities
  in the homogeneous Sobolev spaces, in particular, the $L^2$ regularity of $v$, can be also propagated by the compressible
  Navier-Stokes equations.

  (iii) The result in (ii) of Theorem \ref{LOCAL} also indicates that the
  uniform boundedness of the entropy can be propagated by the compressible
  Navier-Stokes equations, if the initial density decays slows to the vacuum
  at the far field.
\end{remark}

\begin{remark}
\label{remarkeulerlocal}
By the definition of strong solutions, we have the regularities
$$
\sqrt{\varrho_0}v\in L^\infty(0,T;L^2),\quad v_y\in L^2(0,T; H^1),
$$
which implies $v\in L^2(0,T; Lip)$. Define the Euler coordinate as
$$
x=\eta(y,t), \quad\partial_t\eta(y,t)=v(y,t),\quad\eta(y,0)=y.
$$
Noticing that
$$
\partial_t\eta_y=v_y=\partial_tJ,\quad\eta_y(y,0)=1=J(y,0),
$$
we have $\eta_y\equiv J$ on $\mathbb R\times(0,T)$. Recalling that $J$
has uniform positive lower and upper bounds on $\mathbb R\times(0,T)$,
for any fixed $t\in(0,T)$, $\eta$ is reversible in $y$. Therefore, one
can define the density $\rho$, velocity $u$, and pressure $p$, in the
Euler coordinate as
$$
\rho(x,t)=\varrho(y,t),\quad u(x,t)=v(y,t),\quad p(x,t)=\pi(y,t),
$$
where $\varrho(y,t):=\frac{\varrho_0(y)}{J(y,t)}$. We can check that
$(\rho, u, p)$ has appropriate regularities, in particular $u\in
L^1(0,T; Lip)$, and it is a solution to system (\ref{ECNSrho}),
(\ref{ECNSu}), and (\ref{ECNSp}), subject to the initial data
$(\varrho_0, v_0, \pi_0)$; while the uniqueness in the Euler
coordinate can be proven by transforming it to the Lagrangian
coordinate, as $u\in
L^1(0,T; Lip)$, and apply the uniqueness result stated in Proposition \ref{propuni}.
\end{remark}

\begin{theorem}[\textbf{Global well-posedness}]
\label{GLOBAL}
Assume that (H1)--(H2) hold, and that
\begin{equation*}\tag{H4}
\varrho_0\in L^1,\quad \pi_0\in L^1,\quad\varrho_0(y)\geq\frac{A_0}{(1+|y|)^2},\quad\forall y\in \mathbb R,
\end{equation*}
for some positive constant $A_0$.

The following two hold:

(i) There is a unique global strong solution $(J, v, \pi)$ to system (\ref{lcnsJ})--(\ref{lcnspi}), subject to the initial condition
$(J,v,\pi)|_{t=0}=(1,v_0,\pi_0).$
Moreover, we have the following
\begin{eqnarray*}
&\displaystyle\int_{\mathbb R}\left(\frac12{\varrho_0(y)v^2(y,t)} +\frac{1}{\gamma-1}J(y,t)\pi(y,t)\right) dy=\mathcal E_0,\\
&\displaystyle\int_\mathbb R\varrho_0(y)v(y,t)dy=m_0,\quad\inf_{y\in\mathbb R}J(y,t)\geq c_0,
\end{eqnarray*}
for any $t\in[0,\infty)$,
where
$$
\mathcal E_0:=\int_{\mathbb R}\left(\frac{\varrho_0v_0^2}{2}+\frac{\pi_0}{\gamma-1}\right) dy,
\quad m_0=\int_{\mathbb R}\varrho_0v_0dy,
\quad c_0=e^{-\frac{2\sqrt{2}}{\mu}\sqrt{\mathcal E_0}\|\varrho_0\|_1}.
$$

(ii) Assume further that (H3) holds,
for two positive constants $\delta$ and $K_0$. Then, (\ref{ADR1}) and (\ref{ADR2}) hold for any $T\in(0,\infty)$.
\end{theorem}

\begin{remark}
If the initial data has more regularities, then the corresponding
solution $(\varrho, v, \pi)$ in Theorem \ref{GLOBAL} can be classical
ones and, consequently, we obtain the global existence of classical
solutions to the compressible Navier-Stokes equations without heat
conduction, in the presence of far field vacuum. To the best of our
acknowledge, this is the first result on the global existence of
strong solutions to the compressible Navier-Stokes equations without
heat conduction, for arbitrary large initial data, in the presence
of far field vacuum.
Note that this
global existence result does not contradict to the finite time blow-up
results in \cite{XINYAN13}, as the assumption that having
initial isolated mass group there is excluded in our case.
\end{remark}

\begin{remark}
\label{remarkrholowerbound}
The following assumption in (H4)
$$
\varrho_0(y)\geq\frac{A_0}{(1+|y|)^2},\quad\forall y\in \mathbb R
$$
can be removed. In fact, noticing that, essentially,
the role that this assumption
played in the proof of Theorem \ref{GLOBAL} is to justify some
integration by parts of some integrals defined
on the whole line, so that
one can get the basic energy inequality and the estimates on $G$,
see Proposition \ref{basic}, Proposition \ref{estG}, and Proposition
\ref{estGextra}.
Alternatively, to get the desired basic energy inequality and
the estimates on $G$, one can approximate
the Cauchy problem by a sequence of initial-boundary value problems,
while for the initial-boundary value problems, the integration by
parts to the corresponding integrals, defined on the finite intervals,
holds without the above assumption.
\end{remark}

\begin{remark}
\label{remarkinitialdataglobal}
Choose
$$
\varrho_0(y)=\frac{K_\varrho}{\langle y\rangle^{\ell_\varrho}},\quad J_0\equiv1,\quad v_0\in C_c^\infty,\quad
\pi_0=Ae^{\frac{1}{c_v}}\varrho_0^\gamma,
$$
where $K_\varrho$ and $\ell_\varrho$ are positive numbers.

(i) If $1<\ell_\varrho\leq2$, then $(\varrho_0,v_0,\pi_0)$ satisfies
assumptions (H1), (H2), (H3), with $\delta=1$, and (H4). Therefore, by Theorem \ref{GLOBAL}, there
is a unique global strong solution $(J,v,\pi)$, with $v$ being in the inhomogeneous Sobolev space $L^\infty(0,T; H^1)$; if moreover that $\gamma>\frac54$ and $\max\left\{1,\frac{1}{2(\gamma-1)}\right\} <\ell_\varrho\leq2$, then $\vartheta_0\in H^1$, and, consequently, the temperature $\vartheta$
also lies in the inhomogeneous Sobolev space $L^\infty(0,T; H^1)$.

(ii) If $\max\left\{1,\frac1\gamma\right\}<\ell_\varrho\leq2$, then $(\varrho_0,v_0,\pi_0)$ satisfies conditions (H1), (H2), (H3), with $\delta=\gamma$, (H4), and $s_0\equiv1$. Therefore, by Theorem \ref{GLOBAL}, there
is a unique strong solution $(J,v,\pi)$, and the corresponding entropy $s$ is uniformly bounded on $\mathbb R\times(0,T)$.

(iii) Combining (i) with (ii), if $\gamma>\frac54$ and $\max\left\{1,\frac1\gamma,\frac{1}{2(\gamma-1)}\right\}<\ell_\varrho\leq2$, then there is a unique global strong solution $(J,v,\pi)$, with the properties that the corresponding entropy is uniformly bounded, and the velocity and the corresponding temperature lie in the inhomogeneous space $L^\infty(0,T; H^1)$.
\end{remark}

\begin{remark}
\label{remarkeulerglobal}
Same as in Remark \ref{remarkeulerlocal}, one can obtain the corresponding global existence of solutions in the Euler coordinates
to the compressible Navier-Stokes equations (\ref{ECNSrho}), (\ref{ECNSu}), and (\ref{ECNSp}), subject to the initial data $(\varrho_0, v_0, \pi_0)$.
\end{remark}

%\begin{remark}
%(i) If replacing the assumptions that $\inf_{y\in(-R,R)}\varrho_0(y)>0$, for any $R\in(0,\infty)$, and $\varrho_0\leq\bar\varrho<\infty$, on $\mathbb R$, in Theorem \ref{thmmain}, by the weaker one $0\leq\varrho_0(y)\leq\bar\varrho<\infty$, on $\mathbb R$, then (i) of Theorem \ref{thmmain} still holds.
%
%(ii) We establish in (ii) of Theorem \ref{thmmain} the local well-posedness of strong
%solutions to system (\ref{LCNSJ})--(\ref{LCNSpi}) in the inhomogeneous spaces.
%
%(iii) Theorem \ref{thmmain} presents the existence of solutions,
%with bounded entropy, to the one-dimensional compressible Navier-Stokes equations, in the presence vacuum at the far field only.
%\end{remark}

\section{Local existence in the absence of vacuum}
\label{seclocnov}
In this section, we study system (\ref{lcnsJ})--(\ref{lcnspi}),
subject to (\ref{IC}), in the absence of vacuum, that is, the density $\varrho_0$ is assumed to have a positive lower bound. We focus
on those a priori estimates of the solutions which
are independent of the positive
lower bound of the density $\varrho_0$.

Then the following local existence result holds:

\begin{proposition}\label{prop3.1}
Given a function $\varrho_0$ satisfying
$\underline\varrho\leq\varrho_0\leq\bar\varrho$ on $\mathbb R$,
for two positive constants $\underline\varrho$ and $\bar\varrho$.
Assume that the initial data $(J_0, v_0, \pi_0)$ satisfies
\begin{eqnarray*}
&\underline J\leq J_0\leq\bar J\mbox{ on }\mathbb R, \quad J_0'\in L^2, \quad v_0\in H^1,\quad0\leq\pi_0\in H^1,
\end{eqnarray*}
for two positive constants $\underline J$ and $\bar J$.

Then, there is a unique local strong
solution $(J, v, \pi)$ to system (\ref{lcnsJ})--(\ref{lcnspi}), subject
to the initial condition (\ref{IC}), on $\mathbb R\times(0,T)$,
satisfying
\begin{eqnarray*}
&&\frac{3}{4}\underline J\leq J\leq\frac54\bar J,\quad\pi\geq0,\mbox{ on }\mathbb R\times[0,T],\\
&&J-J_0\in C([0,T]; H^1),\quad J_t\in L^\infty(0,T; L^2),\\
&&v\in C([0,T]; H^1)\cap L^2(0,T; H^2),\quad v_t\in L^2(0,T; L^2),\\
&&\pi\in C([0,T]; H^1),\quad\pi_t\in L^2(0,T; L^2),
\end{eqnarray*}
where $T=T(\mu, \gamma, \underline\varrho, \bar\varrho, \ell_0)$, with
$\ell_0=\frac{1}{\underline J}+\bar J+\|J_0'\|_2+\|(v_0,\pi_0)\|_{H^1}$,
and the existence time
$T$ viewing as a function of $\ell_0$ is continuous
in $\ell_0\in(0,\infty)$.
\end{proposition}

\begin{proof}
Let $T$ be a small positive time to be determined by the quantity
$\ell_0$. Given a velocity $v\in L^\infty(0,T; H^1)\cap L^2(0,T;H^2)$, define $J$ and
$\pi$, successively, as the unique solutions to the following two
ordinary differential equations:
$$
J_t=v_y,
$$
and
$$
\pi_t+\gamma\frac{v_y}{J}\pi=(\gamma-1)\mu\left(\frac{v_y}{J}\right)^2,
$$
with initial data $J|_{t=0}=J_0$ and $\pi|_{t=0}=\pi_0$, respectively.
Then, define $V$ as the unique solution to the following uniform parabolic equation
$$
V_t-\frac{\mu}{\varrho_0}\left(\frac{V_y}{J}\right)_y =-\frac{\pi_y}{\varrho_0},
$$
subject to the initial data $V|_{t=0}=v_0$. Define a solution mapping
$\mathcal M: v\mapsto V$, with $V$ defined as above. By
standard energy estimates, and choosing $T=T(\mu, \gamma, \underline\varrho, \bar\varrho, \ell_0)$ small enough,
one can show that $\mathcal M$ is a contracting mapping on the
space $X=L^\infty(0,T;H^1)\cap L^2(0,T;H^2)$ and, thus, there is
a unique fixed point, denoted by $v$, to $\mathcal M$ in $X$. Then $(J,v,\pi)$, with $(J,\pi)$ defined in the way as stated above,
is a desired solution to system (\ref{lcnsJ})--(\ref{lcnspi}),
subject to (\ref{IC}), on $\mathbb R\times(0,T)$.
Since the proof is lengthy but standard, the details are omitted here.
\end{proof}

By Proposition \ref{prop3.1}, there is a positive time
$T_1$, such that system (\ref{lcnsJ})--(\ref{lcnspi}), subject to (\ref{IC}),
has a unique solution $(J, v, \pi)$, on the time interval $(0,T_1)$,
satisfying
\begin{equation*}
\left\{
\begin{array}{l}
\frac{3}{4}\underline J\leq J\leq\frac54\bar J,\quad\pi\geq0,\quad\mbox{ on }\mathbb R\times[0,T_1], \\
J-J_0\in C([0,T_1]; H^1),\quad J_t\in
L^\infty(0,T_1; L^2),\\
v\in C([0,T_1]; H^1)\cap L^2(0,T_1; H^2),\quad v_t\in L^2(0,T_1; L^2),\\
\pi\in C([0,T_1]; H^1),\quad\pi_t\in L^2(0,T_1; L^2),
\end{array}
\right.
\end{equation*}
where $T_1$ is a positive constant depending only on $\mu, \gamma,
\underline\varrho, \bar\varrho,$ and $\frac{1}{\underline J}+\bar J+\|J_0'\|_2+\|(v_0,\pi_0)\|_{H^1}$.
Starting from the time $T_1$, noticing that
$(J, v, \pi)|_{t=T_1}$ satisfies the
conditions on the initial data stated in Proposition \ref{prop3.1},
one can extend the solution $(J, v, \pi)$
forward in time to another time $T_2=T_1+t_1$, for some positive time
$$
t_1=t_1(\mu, \gamma, \underline\varrho, \bar\varrho,
\ell(T_1)),
$$
where, for simplicity of notations, we have denoted
\begin{equation}\label{ell}
\ell(t):=\left(\left(\inf_{y\in\mathbb R}J\right)^{-1}+\|J\|_\infty+\|J_y\|_2+\|v\|_{H^1}+\|\pi\|_{H^1} \right)(t),
\end{equation}
such that $(J, v, \pi)$ is the unique solution to system
(\ref{lcnsJ})--(\ref{lcnspi}), subject to (\ref{IC}), on the time internal
$(0,T_2)$, and that it enjoys the same regularities as above in the time
interval $(0,T_2)$, and $(\frac{3}{4})^2\underline J\leq J\leq (\frac54)^2\bar J$ on $\mathbb
R\times[0,T_2]$. Continuing this procedure, one obtains two sequences of
positive numbers $\{t_j\}_{j=1}^\infty$ and $\{T_j\}_{J=1}^\infty$, with
$$
t_j=t_j(\mu,\gamma,\underline\varrho,\bar\varrho,\ell(T_j)),
$$
and $T_{j+1}=T_j+t_j$, such that the solution $(J, v, \pi)$ can be extended
to time intervals $(0,T_j)$, satisfying
\begin{equation*}
\left\{
\begin{array}{l}
(\frac34)^j\underline J\leq J\leq(\frac54)^j\bar J,\quad \pi\geq0,\quad\mbox{ on }\mathbb R\times[0,T_j],\\
J-J_0\in C([0,T_j]; H^1),\quad J_t\in
L^\infty(0,T_j; L^2),\\
v\in C([0,T_j]; H^1)\cap L^2(0,T_j; H^2),\quad v_t\in L^2(0,T_j; L^2),\\
\pi\in C([0,T_j]; H^1),\quad\pi_t\in L^2(0,T_j; L^2),
\end{array}
\right.
\end{equation*}
for $j=1,2,\cdots.$ Set the maximal existing time $T_\infty$ as
$$
T_\infty=T_1+\sum_{j=1}^\infty t_j.
$$
Then, the solution $(J, v, \pi)$ can be extended to the time interval
$(0,T_\infty)$, such that
\begin{equation}\label{reg}
\left\{
\begin{array}{l}
\displaystyle 0<\inf_{y\in\mathbb R, t\in[0,T]}J(y,t)\leq\sup_{y\in\mathbb R,
t\in[0,T]}J(y,t)<\infty,\\
\pi\geq0\quad \mbox{on }\mathbb R\times[0,T],\\
J-J_0\in C([0,T]; H^1),\quad J_t\in L^\infty(0,T; L^2),\\
v\in C([0,T]; H^1)\cap L^2(0,T; H^2),\quad v_t\in L^2(0,T; L^2),\\
\pi\in C([0,T]; H^1),\quad\pi_t\in L^2(0,T; L^2),
\end{array}
\right.
\end{equation}
for any $T\in(0,T_\infty)$. Moreover, if $T_\infty<\infty$, it must have $\lim_{j\rightarrow\infty}t_j=0$ and, consequently, one has
\begin{equation}
  \varlimsup_{j\rightarrow\infty}\ell(T_j)=\infty, \label{crit}
\end{equation}
where $\ell(t)$ is defined by (\ref{ell}); otherwise, if (\ref{crit}) is not true, then $\ell(T_j)$ is uniformly bounded and, thus, by Proposition \ref{prop3.1}, $t_j=t_j(\mu,\gamma,\underline\varrho,\bar\varrho,\ell(T_j))$, $j=1,2,\cdots,$ have a uniform positive lower bound, contradicting to the fact that $\lim_{j\rightarrow\infty}t_j=0$.

Thanks to the statements in the above paragraph, in the rest of this section,
we always assume, without any further mention, that
$(J,v,\pi)$ is the unique solution to system (\ref{lcnsJ})--(\ref{lcnspi}), subject to (\ref{IC}), and that it
has been extended, in the same way as above, to the maximal existing
time interval $(0,T_\infty)$, where the maximal time $T_\infty$ is constructed in the same way as above.

To obtain the a priori estimates on $(J, v, \pi)$, we define a positive time
\begin{equation}\label{Ts}
T_s:=\sup\left\{T\in(0,T_\infty)~\bigg|~\frac{\underline J}{2}
\leq J\leq 2\bar{J} \mbox{ on }
\mathbb R\times[0,T]\right\}.
\end{equation}

We start with the following estimate on $G$:

\begin{proposition}
\label{propGphi}
There is a positive constant $t_*^1=t_*^1(\gamma,\mu,\bar\varrho,\underline J,\|\sqrt{J_0}G_0\|_2)$, such that 
\begin{eqnarray*}
&\displaystyle\sup_{0\leq t\leq T_*^1} \|\sqrt{J}G\|_2^2 +\mu\int_0^{T_*^1} \left\|
\frac{G_y} {\sqrt{\varrho_0}}\right\|_2^2 dt\leq 3(1+\|\sqrt{J_0}G_0\|_2^2),\\
&\displaystyle\int_0^{T_*^1}\|G\|_\infty^4dt\leq C(\mu,\bar\varrho,\underline J,\|\sqrt{J_0}G_0\|_2),
\end{eqnarray*}
where $G_0:=\mu v_0'-\pi_0$, $T_*^1:=\min\{1,t_*^1, T_s\}$, and $T_s$ is defined by (\ref{Ts}).
\end{proposition}

\begin{proof}
Multiplying (\ref{eqG}) by $JG$, and integrating the resultant over
$\mathbb R$, one gets by integration by parts that
\begin{equation}\label{G.2}
  \int_{\mathbb R}JGG_tdy+\mu\int_{\mathbb
  R}\frac{(G_y)^2}{\varrho_0}dy=-\gamma\int_{\mathbb R}
  v_y G^2dy.
\end{equation}
Then (\ref{lcnsJ}) shows
\begin{equation*}
  \int_{\mathbb R}JGG_tdy =  \frac12\left(\frac{d}{dt}\int_{\mathbb R}JG^2dy- \int_{\mathbb R}v_y
  G^2dy\right),
\end{equation*}
which, together with (\ref{G.2}), yields
\begin{eqnarray}\label{G.3}
   \frac12\frac{d}{dt}\|\sqrt JG\|_2^2+\mu\left\|\frac{G_y}{\sqrt\varrho_0}\right\|_2^2
   =\left(\frac12-\gamma\right)\int_{\mathbb R} v_yG^2 dy.
\end{eqnarray}
It follows from $v_y= \frac{J}{\mu}(G+\pi)$ that
\begin{eqnarray*}
  \left(\frac12-\gamma\right)\int_{\mathbb R} v_yG^2 dy
  &=&\frac{1-2\gamma}{2\mu}\int_{\mathbb R}J(G+\pi)G^2 dy\\
  &\leq& \frac{1-2\gamma}{2\mu}\int_{\mathbb R}JG^3 dy
  \leq\frac\gamma\mu\|G\|_\infty\|\sqrt JG\|_2^2,
\end{eqnarray*}
where we have used
$\pi\geq0$ and $\gamma>1$.
Therefore, it follows from (\ref{G.3}) that
\begin{equation}
  \label{M.1}
  \frac12\frac{d}{dt}\|\sqrt JG\|_2^2+\mu\left\|\frac{G_y}{\sqrt\varrho_0}\right\|_2^2\leq  \frac\gamma\mu\|G\|_\infty\|\sqrt JG\|_2^2.
\end{equation}
By the Gagliardo-Nirenberg inequality $\|f\|_{L^\infty(\mathbb R)}\leq
C\|f\|_{L^2(\mathbb R)}^{\frac12}\|f'\|_{L^2(\mathbb R)}^{\frac12}$, one has
\begin{equation}\label{A.1-1}
\|G\|_\infty\leq C\|G\|_2^{\frac12}\|G_y\|_2^{\frac12}\leq C(\bar\varrho,\underline J)
\|\sqrt JG\|_2^{\frac12}\left\|\frac{G_y}{\sqrt{\varrho_0}}\right\|_2^{\frac12}.
\end{equation}
Combining (\ref{M.1}) and (\ref{A.1-1}), one obtains
from the Young inequality that
\begin{eqnarray*}
\frac12\frac{d}{dt}\|\sqrt JG\|_2^2+\mu\left\|\frac{G_y}{\sqrt\varrho_0}\right\|_2^2&\leq & C(\gamma,\mu,\bar\rho,\underline J)\|\sqrt JG\|_2^{\frac52}\left\|\frac{G_y}{\sqrt{\varrho_0}}\right\|_2^{\frac12}\\
  &\leq&\frac\mu2\left\|\frac{G_y}{\sqrt{\varrho_0}} \right\|_2^2
  +C(\gamma,\mu,\bar\rho,\underline J)(1+\|\sqrt JG\|_2^2)^2
\end{eqnarray*}
and, thus,
\begin{equation}
  \frac{d}{dt}(1+\|\sqrt
  JG\|_2^2)+\mu\left\|\frac{G_y}{\sqrt{\varrho_0}} \right\|_2^2
  \leq C_1(\gamma,\mu,\bar\rho,\underline J)(1+\|\sqrt JG\|_2^2)^2.\label{A.2}
\end{equation}
Solving (\ref{A.2}) yields
\begin{eqnarray*}
  -(1 +\|\sqrt JG\|_2^2)^{-1}(t)&\leq& -(1
  +\|\sqrt J_0G_0\|_2^2)^{-1}+C_1\left(\gamma, \mu, \bar\varrho,\underline J\right)t\\
  &\leq&-\frac{1}{2}(1+\|\sqrt{J_0}G_0\|_2^2)^{-1},
\end{eqnarray*}
for any $t\in[0,T_1^*)$, where
\begin{eqnarray}
  &T_*^1:=\min\{1,t_*^1,T_s\},
  \quad t_*^1:=\frac{1}{2(1+\|\sqrt{J_0}G_0\|_2^2)C_1\left(\gamma, \mu, \bar\varrho,\underline J\right)}.
\end{eqnarray}
Therefore, we have
\begin{equation*}
  \label{A.4}
  \sup_{0\leq t<T_*^1}(1 +\|\sqrt JG\|_2^2)\leq 2(1+\|\sqrt{J_0} G_0\|_2^2),
\end{equation*}
and further from (\ref{A.2}) that
\begin{eqnarray*}
  &&\sup_{0\leq t< T_*^1} \|\sqrt JG\|_2^2 + \mu \int_0^{T_*^1}\left\|\frac{G_y}{\sqrt{\varrho_0}}
  \right\|_2^2 dt
  \leq 3(1+\|\sqrt{J_0}G_0\|_2^2).
\end{eqnarray*}
The estimate $\int_0^{T_*^1}\|G\|_\infty^4dt\leq C(\mu,\bar\varrho,\underline J, \|G_0\|_2)$ follows from the above estimate and (\ref{A.1-1}). The proof is complete.
\end{proof}

Based on Proposition \ref{propGphi}, we can derive the following estimates on $(J,v,\pi)$.

\begin{proposition}
  \label{proppi}
(i) Let $T_*^1$ be as in Proposition \ref{propGphi}. Then, it holds that
\begin{eqnarray*}
\sup_{0\leq t\leq T_*^1}\left\|\left(\pi, \frac{\pi_y}{\sqrt{\varrho_0}},J-J_0,J_t, \frac{J_y}{\sqrt{\varrho_0}}, v_y\right)\right\|_2&\leq& C,\\
\int_0^{T_*^1}
  \left(\|\pi_t\|_2^4+\left\|\left(\sqrt{\varrho_0}v_t,\frac{v_{yy}} {\sqrt{\varrho_0}}\right)\right\|_2^2\right)dt&\leq& C,
\end{eqnarray*}
and
\begin{eqnarray*}
\sup_{0\leq t\leq T_*^1}\|\sqrt{\varrho_0}v\|_{L^2((-R,R))}\leq \|v_0\|_{L^2((-R,R))} +C,
\end{eqnarray*}
for a positive constant $C$ depending only on $\gamma,\mu,\bar\varrho,\underline J,\bar J,\big\|\frac{J_0'}{\sqrt{\varrho_0}}\big\|_2, \|G_0\|_2,\|\pi_0\|_2,$ and $\left\|\frac{\pi_0'} {\sqrt{\varrho_0}}\right\|_2$, but independent of $\underline\varrho$.

(ii) There is a positive constant $t_*^2$ depending only on $\gamma, \mu, \bar\varrho, \underline J, \bar J, \|G_0\|_2,$ and $\|\pi_0\|_\infty$, but independent of $\underline\varrho$, such that
$$
\frac34\underline J\leq J(y,t)\leq\frac54\bar J,\quad \mbox{ on }\mathbb R\times[0,T_*^2),
$$
where $T_*^2:=\min\{T_*^1,t_*^2\}=\min\{1,t_*^1, t_*^2, T_s\}$, with $T_s$ defined by (\ref{Ts}).
\end{proposition}

\begin{proof}
(i) Equation (\ref{lcnspi}) can be rewritten in terms of $G$ as
\begin{equation}
  \pi_t+\frac1\mu\left(\pi+\frac{2-\gamma}{2} G \right)^2= \frac{\gamma^2}{4\mu} G ^2, \label{pi.1}
\end{equation}
from which, one obtains
$$
0\leq\pi(y,t)\leq\pi_0(y)+\frac{\gamma^2}{4\mu}\int_0^t G^2(y,\tau)d\tau.
$$
Thanks to the above, it follows from Proposition \ref{propGphi} and the
H\"older inequality that
\begin{equation}
\sup_{0\leq t\leq T_*^1}\|\pi\|_2 \leq  \|\pi_0\|_2+
\frac{\gamma^2}{4\mu}\int_0^{T_*^1}\|G\|_\infty\|G\|_2dt
\leq\|\pi_0\|_2+C\left(\gamma,\mu,\bar\varrho,\underline J,\bar J,\|G_0\|_2\right)
 \label{pi.2}
\end{equation}
and
\begin{equation}\label{pi.3}
\sup_{0\leq t\leq T_*^1}\|\pi\|_\infty \leq  \|\pi_0\|_\infty+
  \frac{\gamma^2}{4\mu}\int_0^{T_*^1}\|G\|_\infty^2dt
  \leq \|\pi_0\|_\infty+C\left(\gamma,\mu,\bar\varrho, \underline J,\bar J,\|G_0\|_2\right).
\end{equation}
Hence, it follows from (\ref{pi.1}) and Proposition \ref{propGphi} that
\begin{eqnarray}
\int_0^{T_*^1}\|\pi_t\|_2^4dt&\leq& C \int_0^{T_*^1}(\|G\|_\infty^4+\|\pi\|_\infty^4)
(\|G\|_2^4+\|\pi\|_2^4) dt\nonumber\\
&\leq& C(\gamma,\mu,\bar\rho,\underline J,\bar J, \|G_0\|_2, \|\pi_0\|_2, \|\pi_0\|_\infty).
\label{pi.3-1}
\end{eqnarray}
Differentiating equation (\ref{pi.1}) with respect to $y$ yields
\begin{equation*}
  \pi_{yt}+\frac2\mu\left(\pi+\frac{2-\gamma}{2} G \right)
  \left(\pi_y+\frac{2-\gamma}{2} G_y \right)= \frac{\gamma^2}{2\mu} G G_y.
\end{equation*}
Multiplying the above equation by $\frac{\pi_y}{\varrho_0}$ and integrating over $\mathbb R$, it follows from the H\"older and Cauchy inequalities that
\begin{eqnarray*}
\frac{d}{dt}\left\|\frac{\pi_y}{\sqrt{\varrho_0}} \right\|_2^2&\leq&C(\gamma,\mu)(\|\pi\|_\infty+\|G\|_\infty)
\left(\left\|\frac{\pi_y}{\sqrt{\varrho_0}}\right\|_2+\left\|\frac{G_y} {\sqrt{\varrho_0}}\right\|_2\right) \left\|\frac{\pi_y}{\sqrt{\varrho_0}}\right\|_2\\
&\leq&C(\gamma,\mu)\left[\left\|\frac{G_y} {\sqrt{\varrho_0}}\right\|_2^2+
(1+\|\pi\|_\infty^2+\|G\|_\infty^2) \left\|\frac{\pi_y}{\sqrt{\varrho_0}}\right\|_2^2\right],
\end{eqnarray*}
from which, by Proposition \ref{propGphi}, and (\ref{pi.3}), it follows from  the Gronwall inequality that
\begin{equation}\label{pi.4}
  \sup_{0\leq t\leq T_*^1}\left\|\frac{\pi_y}{\sqrt{\varrho_0}}\right\|_2^2\leq C\left(\gamma,\mu,\bar\varrho,\underline J,\bar J, \|G_0\|_2,\|\pi_0\|_\infty, \left\|\frac{\pi_0'} {\sqrt{\varrho_0}}\right\|_2\right).
\end{equation}

Recalling the definition of $G$, one can rewrite the equation for $J$ as
\begin{equation}
  \label{J.1}
  J_t =\frac{J}{\mu}(G+\pi),
\end{equation}
from which, we deduce
\begin{eqnarray*}
  \|J-J_0\|_2+\|J_t\|_2&=&\left\|\int_0^tJ_td\tau\right\|_2+\|J_t\|_2\\
&\leq&\frac{2\bar J}{\mu}\left(\int_0^t(\|G\|_2+\|\pi\|_2)d\tau +\|G\|_2+\|\pi\|_2\right),
\end{eqnarray*}
and, thus, it follows from Proposition \ref{propGphi} and (\ref{pi.2})
that
\begin{equation}\label{J.1-1}
  \sup_{0\leq t\leq T_*^1}(\|J-J_0\|_2+\|J_t\|_2)\leq
  C(\gamma,\mu,\bar\rho,\underline J,\bar J,\|G_0\|_2,\|\pi_0\|_2).
\end{equation}
Solving the ordinary differential equation (\ref{J.1}) yields
\begin{equation*}
  J(y,t)=\exp\left\{\frac1\mu\int_0^t(G+\pi)d\tau\right\}J_0(y)
\end{equation*}
and, thus,
\begin{equation*}
  J_y=\left(\frac1\mu\int_0^t(G_y+\pi_y) d\tau J_0+J_0'\right)\exp\left\{\frac1\mu\int_0^t(G+\pi)d\tau\right\},
\end{equation*}
from which, applying Propositions \ref{propGphi}, and using (\ref{pi.3})
and (\ref{pi.4}), one obtains
\begin{eqnarray}
  \sup_{0\leq t\leq T_*^1}\left\|\frac{J_y}{\sqrt{\varrho_0}}\right\|_2
  &\leq&\left(\frac{\bar J}{\mu}\int_0^{T_*^1}\left\|\left(\frac{G_y}{\sqrt{\varrho_0}} ,\frac{\pi_y}{\sqrt{\varrho_0}}\right) \right\|_2 dt+\left\|\frac{J_0'}{\sqrt{\varrho_0}}\right\|_2\right)\nonumber\\
  &&\times\exp\left\{\frac1\mu\int_0^{T_*^1}\|(G,\pi)\|_\infty
  dt\right\} \leq C,\label{pi.5}
\end{eqnarray}
for a positive constant $C$ depending only on $\gamma,\mu,\bar\varrho,\underline J, \bar J, \big\|\frac{J_0'}{\sqrt{J_0}}\big\|_2,\|G_0\|_2, \|\pi_0\|_2$, and $\big\|\frac{\pi_0'} {\sqrt{\varrho_0}}\big\|_2$,
but independent of $\underline\varrho$.

Recalling the definition of $G$, and noticing that
$\varrho_0v_t=G_y,$ one gets from Proposition \ref{propGphi} and (\ref{pi.2}) that
\begin{equation}
\sup_{0\leq t\leq T_*^1}\|v_y\|_2^2
=\sup_{0\leq t\leq T_*^1}\left\|\frac{J}{\mu}(G+\pi)\right\|_2^2
\leq C(\gamma,\mu,\bar\varrho,\underline J,\bar J,\|G_0\|_2,\|\pi_0\|_2)\label{pi.5-1}
\end{equation}
and
\begin{equation}\label{pi.6}
 \int_0^{T_*^1}\|\sqrt{\varrho_0}v_t\|_2^2dt
=\int_0^{T_*^1} \left\|\frac{G_y}{\sqrt{\varrho_0}}\right\|_2^2 dt
\leq C(\mu,\bar\varrho,\underline J,\bar J,\|G_0\|_2).
\end{equation}
Therefore, it holds that
\begin{eqnarray*}
\sup_{0\leq t\leq T_*^1}\|\sqrt{\varrho_0}v\|_{L^2((-R,R))}&=&\sup_{0\leq t\leq T_*^1} \left\|\sqrt{\varrho_0}v_0+\int_0^t \sqrt{\varrho_0}v_td\tau\right\|_{L^2((-R,R))}\\
&\leq&\|\sqrt{\varrho_0}v_0\|_{L^2((-R,R))} +\int_0^{T_*^1}\|\sqrt{\varrho_0}v_t\|_2d\tau\\
&\leq&\|\sqrt{\varrho_0}v_0\|_{L^2((-R,R))} +C(\mu,\bar\varrho,\underline J,\bar J,\|G_0\|_2),
\end{eqnarray*}
for any $0<R\leq\infty$. Noticing that
\begin{equation*}
  v_{yy}=\left[\frac{J}{\mu}(G+\pi)\right]_y=\frac{J_y}{\mu}(G+\pi) +\frac{J}{\mu}(G_y+\pi_y)
\end{equation*}
it follows from Proposition \ref{propGphi}, (\ref{pi.3}),
(\ref{pi.4}), (\ref{pi.5}), and the H\"older inequality that
\begin{equation}
  \int_0^{T_0}\left\|\frac{v_{yy}}{\sqrt{\varrho_0}}\right\|_2^2 dt  \leq C\int_0^{T_0}\left(\left\|\left(\frac{G_y}{\sqrt{\varrho_0}} ,\frac{\pi_y}{\sqrt{\varrho_0}}\right)\right\|_2^2+\|(G,\pi)\|_\infty^2 \left\|\frac{J_y}{\sqrt{\varrho_0}}\right\|_2^2\right)dt
\leq C,\label{pi.7}
\end{equation}
for a positive constant $C$ depending only on $\gamma,\mu,\bar\varrho,\underline J, \bar J, \|G_0\|_2,\|\pi_0\|_2,$ and $\left\|\frac{\pi_0'} {\sqrt{\varrho_0}}\right\|_2$, but independent of $\underline\varrho$.

(ii) Due to Proposition \ref{propGphi} and (\ref{pi.3}), one gets
from the H\"older inequality that
\begin{eqnarray*}
  \frac1\mu\int_0^t\|J(G+\pi)\|_\infty d\tau
  &\leq& \frac{2\bar J}{\mu}\int_0^t(\|G\|_\infty+\|\pi\|_\infty) d\tau\\
  &\leq&C t+\frac2\mu\left(\int_0^t\|G\|_\infty^4d\tau \right)^{\frac14} t^{\frac34}\leq C_2t^{\frac34}\leq\frac{\underline J}{4} ,
\end{eqnarray*}
for any $t\in[0,T_*^2)$, where
$$
T_*^2:=\min\{T_*^1,t_*^2\}=\min\{t_*^1,t_*^2,T_s\},\quad t_*^2:=\left(\frac{\underline J}{4C_2}\right)^{\frac43},
$$
for a positive constant $C_2$ depending only on $\gamma, \mu, \bar\varrho,\underline J,\bar J,\|G_0\|_2,$ and $\|\pi_0\|_\infty$, but independent of $\underline\varrho$.
Consequently, it follows from (\ref{J.1}) that
$$
  |J-J_0|=\left|\int_0^tJ_td\tau\right| =\left|\int_0^t\frac{J}{\mu}(G+\pi)d\tau\right| \leq \frac1\mu\int_0^t\|J(G+\pi)\|_\infty d\tau\leq\frac{\underline J}{4},
$$
which implies
$$
\frac34\underline J\leq J_0-\frac{\underline J}{4}\leq J\leq J_0+\frac{\underline J}{4}\leq\bar J+\frac{\bar J}{4}=\frac54\bar J\quad\mbox{ on }\mathbb R\times[0,T_2^*),
$$
which proves (ii).
\end{proof}

Thanks to the estimates stated in Proposition \ref{proppi}, one can evaluate the lower bound of the time $T_s$ as stated in the following
proposition:

\begin{proposition}\label{propestTs}
Let $T_s$ be defined by (\ref{Ts}), $t_*^1$ and $T_*^1$ be the constants stated in Proposition \ref{propGphi}, and $t_*^2$ and $T_*^2$ the constants in Proposition \ref{proppi}. Then, we have $
T_s>T_*^2$ and, consequently, $T_*^1\geq T_*^2\geq\min\{1,t_*^1,t_*^2\}$.
\end{proposition}

\begin{proof}
Assume, by contradiction, that $T_s\leq T_*^2$. Recall that $T_*^2=\min\{1,t_*^1, t_*^2,T_s\}$, which gives $T_*^2\leq T_s$ and, therefore, $T_s=T_*^2$.

If $T_s<T_\infty$, then $T_*^2=T_s<T_\infty$. Recalling that (\ref{reg})
holds for any $T<T_\infty$, we have
$(J-J_0,v,\pi)\in C([0,T_\infty); H^1(\mathbb R)),$
which, by the embedding $H^1(\mathbb R)\hookrightarrow L^\infty(\mathbb
R)$, we have $J\in C([0,T_\infty); L^\infty(\mathbb R))$. Thanks to this
and by (ii) of Proposition \ref{proppi}, there is a positive time
$T_*\in(T_*^2,T_\infty)$, such that $\frac12\underline J
\leq J\leq\frac32\bar J$ on $\mathbb R\times[0,T_*]$. By the definition
of $T_s$, then $T_*\leq T_s=T_*^2$, which contradicts to
$T_*\in(T_*^2,T_\infty)$.

If $T_s=T_\infty$, then $T_s=T_*^2=T_\infty$. By Proposition
\ref{proppi}, and noticing that $T_*^2\leq T_*^1$, we have
\begin{equation*}
  \sup_{0\leq t<T_*^2}(\|J_y\|_2+ \|\sqrt{\varrho_0}v\|_2^2+\|v_y\|_2^2 +\|\pi\|_{H^1})<\infty
\end{equation*}
and $\frac34\underline J\leq J\leq\frac54\bar J$ on
$\mathbb R\times[0,T_*^2)$. Therefore, recalling that
$\varrho_0\geq\underline{\varrho}>0$, we have
$\varlimsup_{t\rightarrow T_\infty}\ell(t)=\varlimsup_{t\rightarrow T_*^2}\ell(t)<\infty,$
which contradicts to (\ref{crit}).

Combining the statements of the above two paragraphes yields 
$T_s>T_*^2$. By the aid of this, and recalling the definition of
$T_*^2=\min\{1,t_*^1,t_*^2,T_s\}$, we have
$T_*^2=\min\{1,t_*^1,t_*^2\}$. This proves the conclusion.
\end{proof}

Then Propositions \ref{prop3.1}--\ref{propestTs} give the following:

\begin{corollary}
  \label{corTextEst}
Given a function $\varrho_0$ satisfying
$\underline\varrho\leq\varrho_0\leq\bar\varrho$ on $\mathbb R$,
for two positive constants $\underline\varrho$ and $\bar\varrho$.
Assume that the initial data $(J_0, v_0, \pi_0)$ satisfies
\begin{eqnarray*}
&\underline J\leq J_0\leq\bar J\mbox{ on }\mathbb R, \quad J_0'\in L^2, \quad v_0\in L^1_{loc},\quad v_0'\in L^2,\quad0\leq\pi_0\in H^1,
\end{eqnarray*}
for two positive constants $\underline J$ and $\bar J$.

Then, there is a positive time $T_0$ depending only on $\gamma,\mu,\bar\varrho,\underline J,\bar J,\|v_0'\|_2$, $\|\pi_0\|_2$, and $\|\pi_0\|_\infty$, but independent of $\underline\varrho$, such that system (\ref{lcnsJ})--(\ref{lcnspi}), subject to (\ref{IC}), has a unique solution $(J, v,\pi)$ on $\mathbb R\times[0,T_0]$,
satisfying
\begin{align*}
&\pi\geq0, \quad \frac34\underline J\leq J\leq\frac54\bar J,\quad\mbox{on }\mathbb R\times[0,T_0],\\
&\sup_{0\leq t\leq T_0}\left\|\left(\pi, \frac{\pi_y}{\sqrt{\varrho_0}},J-J_0,J_t, \frac{J_y}{\sqrt{\varrho_0}}, v_y\right)\right\|_2\leq C,\\
&\int_0^{T_0}
  \left(\|\pi_t\|_2^4+\left\|\left(\sqrt{\varrho_0}v_t,\frac{v_{yy}} {\sqrt{\varrho_0}}\right)\right\|_2^2\right)dt  \leq C, \\
&\sup_{0\leq t\leq T_0}\|\sqrt{\varrho_0}v\|_{L^2((-R,R))} \leq \|\sqrt{\varrho_0}v_0\|_{L^2((-R,R))}+C,
\end{align*}
for any $0<R\leq\infty$,
where $C$ is a positive constant depending only on $\gamma,\mu,\bar\varrho,\underline J,\bar J,$ $\big\|\frac{J_0'}{\sqrt{\varrho_0}}\big\|_2,\|v_0'\|_2, \|\pi_0\|_2,$ and $\big\| \frac{\pi_0'}{\sqrt{\varrho_0}}\big\|_2$, but independent of $\underline\varrho$.
\end{corollary}

\begin{proof}
The H\"older inequality yields 
\begin{equation}
  \label{ADEQ1}
  |v_0(y)|=\left|v_0(0)+\int_0^yv_0'(z)dz\right|\leq|v_0(0)|+\|v_0'\|_2 \sqrt{|y|},\quad\forall y\in\mathbb R.
\end{equation}
Choose a function $0\leq\phi\in C_c^\infty((-2,2))$, with $\phi\equiv1$ on
$(-1,1)$, $0\leq\phi\leq1$ on $(-2,2)$, and $|\phi'|\leq 2$ on $\mathbb R$.
For any positive
integer $n$, we set $\phi_n(\cdot)=\phi(\frac{\cdot}{n})$ and $v_{0n}=
v_0\phi_n$. Thanks to (\ref{ADEQ1}), noticing that $supp\,\phi_n\subseteq(-2n,-n)\cup(n,2n)$, and that
$$
v_{0n}'=v_0'\phi_n+v_0\phi_n'=v_0'\phi_n+\frac{v_0}{n}\phi'\left(
\frac \cdot n\right),
$$
one has 
\begin{eqnarray}
  \|v_{0n}'\|_2&\leq&\|v_0'\|_2+\frac1n\left\|v_0\phi'\left(\frac\cdot n
  \right)\right\|_2 \nonumber\\
  &=&\|v_0'\|_2+\frac1n\left(\int_{n<|y|<2n}|v_0|^2\left|\phi'\left(\frac yn\right)\right|^2dy\right)^{\frac12}\nonumber\\
  &\leq&\|v_0'\|_2+\frac2n\left[\int_{n<|y|<2n}(|v_0(0)|+\sqrt{|y|}\|v_0'\|_2)^2
  dy\right]^{\frac12}\nonumber\\
  &\leq&\|v_0'\|_2+\frac2n(|v_0(0)|+\sqrt{2n}\|v_0'\|_2)\sqrt{2n}\nonumber\\
  &=&5\|v_0'\|_2+\frac{2\sqrt 2}{\sqrt n}|v_0(0)|\leq   5\|v_0'\|_2+1,\label{ADEQ2}
\end{eqnarray}
for any $n\geq8|v_0(0)|^2$.

Consider system (\ref{lcnsJ})--(\ref{lcnspi}), subject to the initial
condition
\begin{equation}
(J, v, \pi)|_{t=0}=(J_0, v_{0n}, \pi_0).
\label{ADEQ3}
\end{equation}
Due to (\ref{ADEQ2}), it holds that 
$$
\|G_{0n}\|_2\leq\mu\|v_{0n}'\|_2+\|\pi_0\|_2\leq \mu (5\|v_0'\|_2+1)+\|\pi_0\|_2,
$$
for any $n\geq8|v_0(0)|^2$, where $G_{0n}:=\mu v_{0n}'-\pi_0$.
Thanks to this and Propositions \ref{prop3.1}--\ref{propestTs}, there is
a positive time $T_0$ depending only on $\gamma, \mu, \bar\varrho, \underline J, \bar J, \|v_0'\|_2, \|\pi_0\|_2$, and $\|\pi_0\|_\infty$, but independent
of $\underline\varrho$, such that system (\ref{lcnsJ})--(\ref{lcnspi}),
subject to (\ref{ADEQ3}), has a unique solution $(J_n, v_n, \pi_n)$,
on $\mathbb R\times(0,T_0)$, satisfying
\begin{align*}
&\pi_n\geq0, \quad \frac34\underline J\leq J_n\leq\frac54\bar J,\quad\mbox{on }\mathbb R\times[0,T_0],\\
&\sup_{0\leq t\leq T_0}\left\|\left(\pi_n, \frac{\partial_y\pi_n}{\sqrt{\varrho_0}},J_n-J_0,\partial_tJ_n, \frac{\partial_yJ_n}{\sqrt{\varrho_0}}, \partial_yv_n\right)\right\|_2\leq C,\\
&\int_0^{T_0}
  \left(\|\partial_t\pi_n\|_2^4+\left\|\left(\sqrt{\varrho_0}
  \partial_tv_n,\frac{
  \partial_y^2v_n} {\sqrt{\varrho_0}}\right)\right\|_2^2\right)dt  \leq C, \\
&\sup_{0\leq t\leq T_0}\|\sqrt{\varrho_0}v_n\|_{L^2((-R,R))} \leq \|\sqrt{\varrho_0}v_0\|_{L^2((-R,R))}+C,
\end{align*}
for any $n\geq8|v_0(0)|^2$, and
for any $0<R\leq\infty$, where $C$ is a positive constant depending only on
$\gamma, \mu, \bar\varrho, \underline J, \bar J$, $\big\|\frac{J_0'}{\sqrt{\varrho_0}}\big\|_2$, $\|v_0'\|_2, \|\pi_0\|_2,$ and $\big\|\frac{\pi_0'}{\sqrt{\varrho_0}}\big\|_2$, but independent of $\underline\varrho$ and $n\geq8|v_0(0)|^2$.

With the above a priori estimates in hand, one can
apply the Banach-Alaoglu theorem, use Cantor's diagonal arguments,
apply the Aubin-Lions lemma, and make use of the weakly lower semi-continuity
of the norms, to show that there is a subsequence, still denoted by
$(J_n, v_n, \pi_n)$, and a triple $(J,v,\pi)$, which
satisfies the same a priori
estimates as above, such that $(J_n, v_n, \pi_n)$
converges, weakly or weak-* in appropriate spaces, to
$(J,v,\pi)$, and $(J,v,\pi)$ is a solution to system (\ref{lcnsJ})--(\ref{lcnspi}), subject to (\ref{IC}). Since the proof
is very similar to that of (i) of Theorem \ref{LOCAL}, in
the next section, we omit the details here. While the uniqueness is
guaranteed by Proposition \ref{propuni}, in the next section.
\end{proof}

As the end of this section, we give some more estimates on $G$ stated in
the next proposition, which will be the key to obtain the boundedness of
the entropy.

\begin{proposition}
\label{propGex}
In addition to the assumptions in Corollary \ref{corTextEst}, we assume that
$$
\left|\left(\frac{1}{\sqrt{\varrho_0}}\right)'(y)\right|\leq \frac{K_0}{2}, \quad \forall y\in \mathbb R,
$$
for some positive constant $K_0$. Let $T_0$ be the positive constant in Corollary \ref{corTextEst} and $(J, v, \pi)$ the solution stated in Corollary \ref{corTextEst}.

Then, for any $\delta\in(0,\infty)$, there is a positive constant $C$
depending only on
$\gamma, \mu, \bar\varrho,$$\underline J,$ $\bar J,$
$K_0,$ $\big\|\frac{J_0'}{\sqrt{\varrho_0}}\big\|_2,\|v_0'\|_2,\|\pi_0\|_2,
\big\| \frac{\pi_0'}{\sqrt{\varrho_0}}\big\|_2$, and $\left\|\varrho_0^{-\frac\delta2}G_0\right\|_2$, but independent of $\underline\varrho$, such that
\begin{eqnarray*}
\sup_{0\leq t\leq T_0}\left\|\frac{G}{\varrho_0^{\frac\delta2}}\right\|_2^2 +\int_0^{T_0} \left\|\frac{G_y}{\varrho_0^{\frac{\delta+1}{2}}}\right\|_2^2 dt\leq C.
\end{eqnarray*}
\end{proposition}

\begin{proof}
Multiplying (\ref{eqG}) by $\frac{JG}{\varrho_0^\delta}$ and integrating
over $\mathbb R$, one gets from integration by parts that
\begin{eqnarray}
  \int_{\mathbb R}\frac{J G}{\varrho_0^\delta}G_tdy+\mu\int_{\mathbb R}\frac{G_y}{\varrho_0}\left(\frac{G}{\varrho_0^\delta}\right)_ydy
  =-\gamma\int_{\mathbb R} v_y\frac{G^2}{\varrho_0^\delta}dy. \label{G'.1}
\end{eqnarray}
Direct calculations yields
\begin{eqnarray}
\int_{\mathbb R}\frac{G_y}{\varrho_0}\left(\frac{G}{\varrho_0^\delta} \right)_y dy&=&\int_{\mathbb R}\frac{G_y}{\varrho_0}\left(\frac{G_y}{\varrho_0^\delta}
-\delta\frac{\varrho_0'} {\varrho_0}\frac{G}{\varrho_0^\delta}\right)dy\nonumber\\
&=&\left\|\frac{G_y}{\varrho_0^{\frac{\delta+1}{2}}}\right\|_2^2-\delta
\int_{\mathbb R} \frac{\varrho_0'} {\varrho_0}\frac{GG_y}{\varrho_0^{\delta+1}} dy.\label{G'.2}
\end{eqnarray}
It follows from (\ref{lcnsJ}) that
\begin{eqnarray}
  \int_{\mathbb R}\frac{J G}{\varrho_0^\delta}G_tdy
   &=& \frac12\left(\frac{d}{dt}\int_{\mathbb R}\frac{J G^2}{\varrho_0^\delta}dy-\int_{\mathbb R}J_t\frac{G^2}{\varrho_0^\delta}dy\right)\nonumber\\
   &=&\frac12\left(\frac{d}{dt}\left\|\sqrt{\frac{J} {\varrho_0^\delta}}G\right\|_2^2 -\int_{\mathbb R}v_y\frac{G^2}{\varrho_0^\delta}dy\right).\label{G'.2-1}
\end{eqnarray}
Plugging (\ref{G'.2}) and (\ref{G'.2-1}) into (\ref{G'.1}) yields
\begin{equation}
  \frac12\frac{d}{dt}\left\|\sqrt{\frac{J} {\varrho_0^\delta}}G\right\|_2^2+\mu \left\| \frac{G_y}{\varrho_0^{\frac{\delta+1}{2}}} \right\|_2^2
   = \left(\frac12-\gamma\right) \int_{\mathbb R} v_y\frac{G^2}{\varrho_0^\delta}dy+\delta\mu
\int_{\mathbb R} \frac{\varrho_0'} {\varrho_0}\frac{GG_y}{\varrho_0^{\delta+1}} dy. \label{G'.3}
\end{equation}
Due to the assumption
$\left|\left(\frac{1}{\sqrt{\varrho_0}}\right)'\right|\leq \frac{K_0}{2}$, or equivalently $|\varrho_0'|\leq K_0\varrho^{\frac32}_0$, it follows from the Cauchy inequality that
\begin{equation}
  \int_{\mathbb R} \frac{\varrho_0'} {\varrho_0}\frac{GG_y}{\varrho_0^{\delta+1}} dy \leq K_0\int_\mathbb R\left|\frac{G}{\varrho_0^{\frac\delta2}}\right|\left|\frac{G_y}{ \varrho^{\frac{\delta+1}{2}}}\right|dy\leq\frac{1}{2\delta} \left\|\frac{G_y}{ \varrho^{\frac{\delta+1}{2}}}\right\|_2^2+\frac{\delta K_0^2}{2}\left\|\frac{G}{\varrho_0^{\frac\delta2}}\right\|_2^2. \label{G'.4}
\end{equation}
Using the definition of $G$ leads to 
\begin{eqnarray}
  \left(\frac12-\gamma\right) \int_{\mathbb R} v_y\frac{G^2}{\varrho_0^\delta}dy
  &=&\left(\frac12-\gamma\right) \int_{\mathbb R} \frac{J}{\mu}(G+\pi)\frac{G^2}{\varrho_0^\delta}dy\nonumber\\
  &\leq &\frac{1-2\gamma}{2\mu}\int_\mathbb R\frac{JG^3}{\varrho_0^\delta}dy
   \leq\frac\gamma\mu\|G\|_\infty \left\|\sqrt{\frac{J}{\varrho_0^\delta}}G\right\|_2^2, \label{G'.5}
\end{eqnarray}
here, we have used the fact that $\gamma>1$ and $\pi\geq0$.
Plugging (\ref{G'.4}) and (\ref{G'.5}) into (\ref{G'.3}) yields
\begin{eqnarray*}
  \frac{d}{dt}\left\|\sqrt{\frac{J}{\varrho_0^\delta}}G\right\|_2^2 +\mu \left\|\frac{G_y}{\varrho_0^{\frac{\delta+1}{2}}} \right\|_2^2
  &\leq&C(\gamma,\mu,\delta,K_0)(1+\|G\|_\infty) \left(\left\|\sqrt{\frac{J}{ \varrho_0^\delta}}G\right\|_2^2 +1\right),
\end{eqnarray*}
from which, by Corollary \ref{corTextEst} and the Gronwall inequality,
we have
\begin{eqnarray*}
\sup_{0\leq t\leq T_0}\left\|\frac{G}{\varrho_0^{\frac\delta2}}\right\|_2^2 +\int_0^{T_0}\left\|\frac{G_y}{\varrho_0^{\frac{\delta+1}{2}}} \right\|_2^2 dt
\leq  C,
\end{eqnarray*}
for a positive constant $C$ depending only on $\gamma, \mu, \bar\varrho,$$\underline J,$ $\bar J,$
$K_0,$ $\big\|\frac{J_0'}{\sqrt{\varrho_0}}\big\|_2,$ $\|v_0'\|_2,$ $\|\pi_0\|_2$,
$\big\| \frac{\pi_0'}{\sqrt{\varrho_0}}\big\|_2$, and $\left\|\varrho_0^{-\frac\delta2}G_0\right\|_2$, but independent of $\underline\varrho$
\end{proof}

\section{Local existence in the presence of far field vacuum}
\label{seclocv}
In this section, we prove the local existence and uniqueness of strong
solutions to system (\ref{lcnsJ})--(\ref{lcnspi}), subject to
(\ref{IC}), in the presence of far field vacuum,
in other words and, thus, prove Theorem \ref{LOCAL}.

We starts with the uniqueness of the solutions.
\begin{proposition}
  \label{propuni}
Given a function $\varrho_0$ satisfying $\inf_{y\in(-R,R)}\varrho_0(y)>0$, for any $R\in(0,\infty)$, and $\varrho_0\leq\bar\varrho$ on $\mathbb R$, for a positive constant $\bar\varrho$. Let $(J_1, v_1, \pi_1)$ and $(J_2, v_2, \pi_2)$ be two solutions to system
(\ref{lcnsJ})--(\ref{lcnspi}), subject to the same initial data, on $\mathbb R\times(0,T)$, satisfying $c_0\leq J_i\leq C_0$ on $\mathbb R\times(0,T)$, for two positive numbers $c_0$ and $C_0$,
and
\begin{eqnarray*}
&\pi_i\in L^2(0,T; L^\infty),\quad
(\partial_tJ_i,\partial_t\pi_i)\in L^1_{loc}(\mathbb R\times[0,T)),\\
&(\sqrt{\varrho_0}v,\partial_yJ_i,\partial_y\pi_i)\in L^\infty(0,T; L^2),\quad \left(\sqrt{\varrho_0}\partial_tv_i,\partial_yv_i,{\partial_y^2v_i} \right)\in L^2(0,T; L^2),
\end{eqnarray*}
for $i=1,2$. Then $(J_1, v_1, \pi_1)\equiv(J_2, v_2, \pi_2)$ on $\mathbb R\times[0,T]$.
\end{proposition}

\begin{proof}
Define $(J, v, \pi)$ as 
$$
J=J_1-J_2,\quad v=v_1-v_2, \quad \pi=\pi_1-\pi_2.
$$
Then, $(J, v, \pi)$ satisfies
\begin{eqnarray}
&&\partial_tJ=\partial_yv,\label{D1}\\
&&\varrho_0\partial_t v-\mu\partial_y\left(\frac{\partial_y v}{J_1}\right)+\partial_y\left(\pi+\frac{\alpha}{J_1} {J} \right)=0,\label{D2}\\
&&\partial_t\pi+\gamma\beta\pi
=\chi\left( {\partial_y v} -\alpha {J} \right),\label{D3}
\end{eqnarray}
where $\alpha=\alpha(y,t), \beta=\beta(y,t),$ and $\chi=\chi(y,t)$ are given functions as follows
$$
\alpha(y,t)=\frac{\partial_y v_2}{J_2},\quad\beta(y,t)=\frac{\partial_y v_1}{J_1},\quad\chi(y,t)=\left[(\gamma-1)\mu\left(\frac{\partial_y v_1}{J_1^2}+\frac{\partial_y v_2}{J_1J_2}\right)-\gamma\frac{\pi_2}{J_1}\right].
$$
Due to the regularities of $(J_i, v_i, \pi_i), i=1,2,$, it holds that
\begin{equation}\label{D4}
(\alpha,\beta)\in L^2(0,T; L^2),\quad (\alpha, \beta, \chi)\in L^2(0,T;L^\infty).
\end{equation}

Choose a function $\eta\in C_c^\infty((-2,2))$, with $\eta\equiv1$ on $(-1,1)$, and $0\leq\eta\leq1$ on $(-2,2)$. For each $r\geq1$, we set $\eta_r(y)=\eta(\frac yr)$, for $y\in\mathbb R$.
Multiplying equations (\ref{D1}), (\ref{D2}), and (\ref{D3}), respectively,
by $J\eta_r^2$, $v\eta_r^2$, and $\pi\eta_r^2$, summing the resultants up,
and integrating over $\mathbb R$, one gets from integration by parts that
\begin{eqnarray*}
  &&\frac12\frac{d}{dt}\int_\mathbb R(J^2+\varrho_0v^2+\pi^2)\eta_r^2dy
  +\mu\int_\mathbb R\frac{(\partial_yv)^2}{J_1}\eta_r^2dy\nonumber \\
  &=&\int_\mathbb R\left[\partial_yvJ+\left(\pi+\frac{\alpha}{J_1}J\right) \partial_yv+\chi(\partial_yv-\alpha J)\pi-\gamma\beta\pi^2\right]\eta_r^2 dy\nonumber\\
  &&+2\int_\mathbb R\left(\pi+\frac{\alpha}{J_1}J-\mu\frac{\partial_yv}{J_1}\right) v\eta_r\eta_r'dy.
\end{eqnarray*}
Note that 
\begin{eqnarray*}
 &&\int_\mathbb R\left[\partial_yvJ+\left(\pi+\frac{\alpha}{J_1}J\right) \partial_yv+\chi(\partial_yv-\alpha J)\pi-\gamma\beta\pi^2\right]\eta_r^2 dy\\
 &\leq&\frac\mu2\int_\mathbb R\frac{(v_y)^2}{J_1}\eta_r^2dy +C\int_\mathbb R
 (J^2+\pi^2+\alpha^2 J^2+\chi^2\pi^2+|\beta|\pi^2)\eta_r^2dy \\
 &\leq&\frac\mu2\int_\mathbb R\frac{(v_y)^2}{J_1}\eta_r^2dy +C(1+\|(\alpha,\beta,\chi)\|_\infty^2)\int_\mathbb R
 (J^2+\pi^2)\eta_r^2dy,
\end{eqnarray*}
so 
\begin{eqnarray*}
  \frac{d}{dt}\int_\mathbb R(J^2+\varrho_0v^2+\pi^2)\eta_r^2dy
  &\leq& C(1+\|(\alpha,\beta,\chi)\|_\infty^2)\int_\mathbb R
 (J^2+\pi^2)\eta_r^2dy\\
 &&+C\int_\mathbb R(|\pi|+|\alpha||J|+|\partial_yv|)|v||\eta_r'|dy.
\end{eqnarray*}

It follows from this, the Gronwall inequality, and (\ref{D4}) that 
\begin{eqnarray*}
  \sup_{0\leq t\leq T}\int_\mathbb R(J^2+\varrho_0v^2+\pi^2)\eta_r^2dy
  \leq C\int_0^T\int_\mathbb R(|\pi|+|\alpha||J|+|\partial_yv|)|v||\eta_r'|dy dt=:Q_r.
\end{eqnarray*}
In order to prove the conclusion, it suffices to show that $Q_r$ tends to zero as $r\rightarrow\infty$.

Note that
$$
v(y,t)=v(z,t)+\int_z^y\partial_yv(y',t)dy',\quad 0<z<1<y<\infty. 
$$
Integrating the above identity with respect to $z$ over the interval $(0,1)$,
and denoting 
$D:=\sup_{y\in(-1,1)}\frac{1}{\varrho_0(y)}$, one obtains by the
H\"older inequality that
\begin{eqnarray*}
  |v(y,t)|&=&\left|\int_0^1v(z,t)dz+\int_0^1\int_z^y\partial_yv(y',t) dy'dz\right|\\
  &\leq&\int_0^1|v|dz+\int_0^y|\partial_yv|dy'
  \leq\frac{1}{\sqrt D}\int_0^1|\sqrt{\varrho_0}v|dz+\int_0^y|\partial_yv |dy'\\
  &\leq&\frac{1}{\sqrt{D}}\|\sqrt{\varrho_0}v\|_2+\sqrt y\|\partial_yv\|_2
  \leq C\sqrt y,\quad\forall y\geq1.
\end{eqnarray*}
In the same way, one has $|v(y,t)|\leq C\sqrt{|y|}$, for any $y\leq-1$ and, thus,
\begin{equation}\label{DD1}
  |v(y,t)|\leq C\sqrt{|y|},\quad \forall y\in(-\infty,-1]\cup[1,\infty).
\end{equation}

Due to (\ref{D1}), one has 
$$
\sup_{0\leq t\leq T}\|J\|_2\leq\int_0^T\|\partial_yv\|_2d\tau <\infty,
$$
thanks to which, by solving (\ref{D3}) as
$$
\pi=e^{-\gamma\int_0^t\beta d\tau}\int_0^te^{\gamma\int_0^\tau\beta d\tau'} \chi(\partial_yv-\alpha J)d\tau,
$$
and recalling (\ref{D4}), one obtains from the Cauchy inequality that
\begin{eqnarray}
  \sup_{0\leq t\leq T}\|\pi\|_2&\leq&e^{\gamma\int_0^T\|\beta\|_\infty dt} \int_0^T\|\chi\|_\infty(\|\partial_yv\|_2+\|\alpha\|_\infty\|J\|_2
  ) dt\nonumber\\
  &\leq&C\int_0^T(\|\chi\|_\infty^2 +\|\partial_yv\|_2^2+\|\alpha\|_\infty^2+1)dt<\infty. \label{DD2}
\end{eqnarray}

Thanks to (\ref{DD1}) and (\ref{DD2}), noticing that $supp\,\eta_r'\subseteq
(-2r,-r)\cup(r,2r)$, and $|\eta_r'|\leq\frac Cr$, one obtains by the H\"older
inequality that
\begin{eqnarray*}
  Q_r&=&C\int_0^T\int_\mathbb R(|\pi|+|\alpha||J|+|\partial_yv|)|v||\eta_r'|dy dt\\
 &=&\int_0^T\int_{r<|y|<2r}(|\pi|+|\alpha||J|+|\partial_yv|)|v||\eta_r'|dydt\\
 &\leq&\frac{C}{\sqrt{r}}\int_0^T\int_{r<|y|<2r}(|\pi|+|\alpha|+|\partial_yv|)dydt\\
 &\leq&C\left(\int_0^T\int_{r<|y|<2r}(|\pi|^2+|\alpha|^2+|\partial_y v|^2)dydt\right)^{\frac12} ,
\end{eqnarray*}
for any $r\geq1$, which, together with the fact that 
$\pi, \alpha, \partial_yv\in L^2(\mathbb R\times(0,T))$, implies
$$
Q_r\rightarrow0,\quad\mbox{as }r\rightarrow\infty.
$$
The proof is complete.
\end{proof}

We are now ready to give the proof of Theorem \ref{LOCAL}.

\begin{proof}[\textbf{Proof of Theorem \ref{LOCAL}.}]
(i) Since the uniqueness is a direct corollary of Proposition \ref{propuni}, it remains to prove the existence.
For any positive number
$\varepsilon\in(0,1)$, set
$\varrho_{0\varepsilon}(y)=\varrho_0(y)+\varepsilon$, for $y\in\mathbb R$.
It is clear that $\varepsilon\leq\varrho_{0\varepsilon}(y)\leq\bar\varrho+1$, for $y\in\mathbb R$.
Consider the following approximate system of (\ref{lcnsJ})--(\ref{lcnspi}):
\begin{equation}
\left\{
\begin{array}{l}
J_t=v_y,\\
\varrho_{0\varepsilon}v_t-\mu\left(\frac{v_y}{J}\right)_y+\pi_y=0,\\
\pi_t+\gamma \frac{v_y}{J}\pi=(\gamma-1)\mu\left(\frac{v_y}{J}\right)^2.
\end{array}
\right.
\label{lcns}
\end{equation}

By Corollary \ref{corTextEst}, there is a positive constant $T$ depending only on $\gamma,\mu, \bar\varrho,\underline J,\bar J,\|v_0'\|_2$, $\|\pi_0\|_2$, and
$\|\pi_0\|_\infty$, but independent of $\varepsilon$, such that system (\ref{lcns}), subject to (\ref{IC}), has a unique solution $(J_\varepsilon, v_\varepsilon, \pi_\varepsilon)$, satisfying
\begin{eqnarray}
&&\pi_\varepsilon\geq0, \quad \frac34\underline J\leq J_\varepsilon \leq\frac54\bar J,\quad\mbox{on }\mathbb R\times[0,T],\label{AP0}\\
&&\sup_{0\leq t\leq T} \left\|\left(J_\varepsilon-J_0,
\frac{\partial_yJ_\varepsilon} {\sqrt{\varrho_{0\varepsilon}}}, \partial_tJ_\varepsilon,\partial_yv_\varepsilon, \pi_\varepsilon,\frac{\partial_y\pi_\varepsilon}{\sqrt{ \varrho_{0\varepsilon}}}\right) \right\|_2^2 \leq C,\label{AP1}\\
&&\sup_{0\leq t\leq T}\|\sqrt{\varrho_{0\varepsilon}}v_\varepsilon\|_{L^2((-R,R))} \leq \|\sqrt{\varrho_{0\varepsilon}}v_0\|_{L^2((-R,R))}+C,\label{AP2}
\end{eqnarray}
for any $0<R<\infty$, and
\begin{equation}\label{AP3}
\int_0^{T} \left(\left\|\left(\frac{ \partial_y^2 v_\varepsilon}{\sqrt{\varrho_{0\varepsilon}}}, \sqrt{\varrho_{0\varepsilon}}\partial_tv_\varepsilon, \right)\right\|_2^2 +\|\partial_t\pi_\varepsilon\|_2^4 \right)dt\leq C,
\end{equation}
for a positive constant $C$ independent of $\varepsilon$.

Due to the a priori estimates (\ref{AP1})--(\ref{AP3}), by the Banach-Alaoglu theorem, and using Cantor's diagonal argument, there is a subsequence, still denoted by $(J_\varepsilon, v_\varepsilon, \pi_\varepsilon)$, and a triple $(J, v, \pi)$, such that
\begin{eqnarray}
J_\varepsilon-J_0\rightarrow J-J_0,&&\mbox{weak-* in }L^\infty(0,T; H^1),\label{APP6}\\
\partial_tJ_\varepsilon\rightarrow J_t,&& \mbox{weak-* in }L^\infty(0,T; L^2), \label{APP7}\\
v_\varepsilon\rightarrow v,&&\mbox{weakly in }L^2(0,T; H^2((-R,R))),\label{APP8}\\
\partial_tv_\varepsilon\rightarrow v_t,&& \mbox{weakly in }L^2(0,T; L^2((-R,R))), \label{APP9}\\
\partial_yv_\varepsilon\rightarrow v_y,&& \mbox{weak-* in }L^\infty(0,T; L^2)\mbox{ and weakly in }L^2(0,T; H^1),\label{APP10}\\
\pi_\varepsilon\rightarrow\pi,&& \mbox{weak-* in }L^\infty(0,T; H^1),\label{APP11}\\
\partial_t\pi_\varepsilon\rightarrow \pi_t,&&\mbox{weakly in }L^4(0,T;L^2),\label{APP12}
\end{eqnarray}
for any $R\in(0,\infty)$, and
\begin{eqnarray}
&&J-J_0\in L^\infty(0,T; H^1),\quad J_t\in L^\infty(0,T; L^2), \label{APP13}\\
&&v_y\in L^\infty(0,T; L^2)\cap L^2(0,T;H^1),\label{APP14} \\
&&\pi\in L^\infty(0,T; H^1),\quad \pi_t\in L^4(0,T; L^2).\label{APP15}
\end{eqnarray}

It remains to prove that $(J,v,\pi)$ is a strong solution to system
(\ref{lcnsJ})--(\ref{lcnspi}), subject to (\ref{IC}), on $\mathbb
R\times(0,T)$. One can verify that $(J,v,\pi)$ has the
regularities stated in Definition \ref{defloc}. Other desired
regularities of
$(J, v, \pi)$, beyond those in (\ref{APP13})--(\ref{APP15}), are
verified as follows. First, thanks to (\ref{APP13}), (\ref{APP15}),
and
$$
Y:=\left\{f|f\in L^\infty(0,T;L^2), f'\in L^1(0,T; L^2)\right\}\hookrightarrow C([0,T]; L^2),
$$
it is clear that $(J-J_0,\pi)\in C([0,T]; L^2).$
Next, noticing that
$\sqrt{\varrho_{0\varepsilon}}\rightarrow
\sqrt{\varrho_0}$ and
$\frac{1}{\sqrt{\varrho_{0\varepsilon}}}\rightarrow
\frac{1}{\sqrt{\varrho_0}}$, in $L^\infty((-R,R))$, for any $R\in(0,\infty)$, one gets from (\ref{APP6}), (\ref{APP8})--(\ref{APP9}), and (\ref{APP11}) that
\begin{eqnarray}
\left(\frac{\partial_yJ_\varepsilon}{\sqrt{\varrho_{0\varepsilon}}},
\frac{\partial_y\pi_\varepsilon}{\sqrt{\varrho_{0\varepsilon}}}\right) \rightarrow\left(\frac{J_y}{\sqrt{\varrho_0}}, \frac{\pi_y}{\sqrt{\varrho_0}}\right),&& \mbox{weak-* in }L^\infty(0,T; L^2((-R,R))),\nonumber\\
\left(\frac{\partial_y^2v_\varepsilon}{\sqrt{\varrho_{0\varepsilon}}},
\sqrt{\varrho_{0\varepsilon}}\partial_tv_\varepsilon\right)
\rightarrow\left(\frac{v_{yy}}{\sqrt{\varrho_0}}, \sqrt{\varrho_0}\partial_tv\right),&& \mbox{weakly in }L^2((-R,R)\times(0,T)),\nonumber
\end{eqnarray}
for any $R\in(0,\infty)$.
Consequently, it follows from the weakly lower semi-continuity of the norms,
(\ref{AP1}), and (\ref{AP3}) that 
\begin{eqnarray*}
  &\left\|\left(\frac{J_y}{\sqrt{\varrho_0}},
  \frac{\pi_y}{\sqrt{\varrho_0}}\right)\right\|_{
  L^\infty(0,T;L^2((-R,R)))} \leq{\displaystyle\varliminf_{\varepsilon\rightarrow0}}
  \left\|\left(\frac{\partial_yJ_\varepsilon} {\sqrt{\varrho_{0\varepsilon}}},
  \frac{\partial_y\pi_\varepsilon}{\sqrt{\varrho_{0\varepsilon}}}\right) \right\|_{L^\infty(0,T;L^2((-R,R)))}\leq C,\\
  &\left\|\left(\frac{v_{yy}}{\sqrt{\varrho_0}}, \sqrt{\varrho_0}\partial_tv\right)\right\|_{L^2(0,T;L^2((-R,R)))}
  \leq{\displaystyle\varliminf_{\varepsilon\rightarrow0}}
  \left\|\left(\frac{\partial_y^2v_\varepsilon}{ \sqrt{\varrho_{0\varepsilon}}},
\sqrt{\varrho_{0\varepsilon}}\partial_tv_\varepsilon\right) \right\|_{L^2(0,T;L^2((-R,R)))}\leq C,
\end{eqnarray*}
for a positive constant $C$ independent of $R$. Therefore,  
$$
\left(\frac{J_y}{\sqrt{\varrho_0}},
\frac{\pi_y}{\sqrt{\varrho_0}}\right)\in L^\infty(0,T; L^2),\quad
   \left(\sqrt{\varrho_0} v_t,\frac{v_{yy}}{\sqrt{\varrho_0}}\right)\in L^2(0,T; L^2).
$$
And finally, since $\sqrt{\varrho_0}v_t\in L^2(0,T; L^2)$,
then $\sqrt{\varrho_0}v\in C([0,T]; L^2)$. Combining all the regularities obtained in the above, we can see that $(J,v,\pi)$ meet the
required regularities in Definition \ref{defloc}.

Next, we show that $\pi\geq0$, $J$ has a uniform positive lower bound on $\mathbb R\times(0,T)$, and that $(J,v,\pi)$ fulfills the
initial condition (\ref{IC}). Thanks to (\ref{APP6})--(\ref{APP9}),  (\ref{APP11})--(\ref{APP12}), the Aubin-Lions compactness lemma, and Cantor's diagonal argument again, there is a subsequence, still denoted by $(J_\varepsilon, v_\varepsilon, \pi_\varepsilon)$, such that
\begin{eqnarray}
&J_\varepsilon\rightarrow J,\quad \mbox{in }C([0,T]; L^2((-R,R))),\label{APP20} \\
&v_\varepsilon\rightarrow v,\quad\mbox{in }C([0,T];L^2((-R,R)))\cap L^2(0,T;H^1((-R,R))),\label{APP21} \\
&\pi_\varepsilon\rightarrow\pi,\quad\mbox{in }C([0,T];L^2((-R,R))),\label{APP22}
\end{eqnarray}
for any $R\in(0,\infty)$. It follows from (\ref{APP20}), (\ref{APP22}), and (\ref{AP0}) that $\pi\geq0$ and $\frac34\underline J
\leq J\leq\frac54\bar J$
on $\mathbb R\times[0,T]$. Moreover, (\ref{APP20})--(\ref{APP22}) guarantees that $(J,v,\pi)$ fulfills the initial condition (\ref{IC}).

And finally, we prove that $(J,v,\pi)$ satisfies equations (\ref{lcnsJ})--(\ref{lcnspi}). Thanks to (\ref{APP6})--(\ref{APP12}) and (\ref{APP20})--(\ref{APP22}), by taking
$\varepsilon\rightarrow0^+$ to system (\ref{lcns}), one can see that
$(J,v,\pi)$ satisfies equations (\ref{lcnsJ})--(\ref{lcnspi}).
Therefore, $(J,v,\pi)$ is a
strong solution to system (\ref{lcnsJ})--(\ref{lcnspi}),
subject to (\ref{IC}), on $\mathbb R\times(0,T)$, which proves (i).

(ii) We first prove the regularities of $G$, i.e., (\ref{ADR1}).
Let $\varrho_{0\varepsilon}, (v_\varepsilon, J_\varepsilon, \pi_\varepsilon),$ and $T$ be the same as in (i). Then,
$\frac34\underline J
\leq J_\varepsilon\leq\frac54\bar J$ on $\mathbb R\times(0,T)$,
and (\ref{APP6})--(\ref{APP12}) and (\ref{APP20})--(\ref{APP22})
hold. It is clear that
\begin{eqnarray*}
\left|\left(\frac{1}{\sqrt{\varrho_{0\varepsilon}}}\right)'\right|
=\frac12\left|\frac{\varrho_{0\varepsilon}'} {\varrho_{0\varepsilon}^{\frac32}}\right|
=\frac12\left|\frac{\varrho_{0}'} {\varrho_{0\varepsilon}^{\frac32}}\right|
\leq\frac12\left|\frac{\varrho_{0}'} {\varrho_{0}^{\frac32}}\right| =\left|\left(\frac{1}{\sqrt{\varrho_{0}}}\right)'\right|\leq\frac{K_0}{2}, \quad\forall y\in\mathbb R.
\end{eqnarray*}
Therefore, one can apply Proposition \ref{propGex} to get
\begin{eqnarray}
\sup_{0\leq t\leq T}\left\|\frac{G_\varepsilon} {\varrho_{0\varepsilon}^{\frac\delta2}}\right\|_2^2 +\int_0^{T} \left\|\frac{\partial_yG_\varepsilon} {\varrho_{0\varepsilon}^{\frac{\delta+1}{2}}}\right\|_2^2 dt\leq C, \label{APP31}
\end{eqnarray}
for a positive constant $C$ independent of $\varepsilon$, where $G_\varepsilon:=\mu\frac{\partial_yv_\varepsilon} {J_\varepsilon}-\pi_\varepsilon$.

Recalling $\frac34\underline J\leq J_\varepsilon\leq\frac54\bar J$,
and using (\ref{APP6}), (\ref{APP10})--(\ref{APP11}), and
(\ref{APP20})--(\ref{APP21}), one can show that
\begin{eqnarray*}
&&G_\varepsilon\rightarrow G,\quad\mbox{weak-* in }L^\infty(0,T; L^2((-R,R))),\\
&&\partial_yG_\varepsilon\rightarrow G_y,\quad\mbox{weakly in }L^2(0,T; L^2((-R,R))),
\end{eqnarray*}
for any $R\in(0,\infty)$.
Therefore, noticing that $\frac{1}{\varrho_{0\varepsilon}}\rightarrow
\frac{1}{{\varrho_0}}$ in $L^\infty((-R,R))$, for any $R\in(0,\infty)$, it is easily to verify that
\begin{eqnarray}
&&\frac{G_\varepsilon}{\varrho_{0\varepsilon}^{\frac\delta2}}
\rightarrow \frac{G}{\varrho_0^{\frac\delta2}},\quad\mbox{weak-* in }L^\infty(0,T; L^2((-R,R))),\label{APP32}\\
&&\frac{\partial_yG_\varepsilon}{\varrho_{0 \varepsilon}^{\frac{\delta+1}{2}}}
\rightarrow \frac{G_y}{\varrho_0^{\frac{\delta+1}{2}}},\quad\mbox{weakly in }L^2(0,T; L^2((-R,R))),\label{APP33}
\end{eqnarray}
for any $R\in(0,\infty)$.

Due to (\ref{APP32}), (\ref{APP33}), and the weakly lower semi-continuity of the norms, it follows from (\ref{APP31}) that
$\frac{G}{\varrho_0^{\frac\delta2}}$ and $\frac{G_y}{\varrho_0^{\frac{\delta+1}{2}}}$, respectively, are bounded in $L^\infty(0,T; L^2((-r,r)))$ and $L^2(0,T;L^2((-r,r)))$, uniformly in
$r\in(0,\infty)$, and, consequently, it holds that 
$$
\frac{G}{\varrho_0^{\frac\delta2}}\in L^\infty(0,T; L^2),\quad
\frac{G_y}{\varrho_0^{\frac{\delta+1}{2}}}\in L^2(0,T;L^2).
$$
Thanks to these regularities of $G$, it follows from the
Gagliardo-Nirenber inequality $\|f\|_{L^\infty(\mathbb R)}\leq
C\|f\|_{L^2(\mathbb R)}^{\frac12}\|f'\|_{L^2(\mathbb R)}^{\frac12}$, for
$f\in H^1(\mathbb R)$, and the assumption
$\big|\big(\tfrac{1}{\sqrt{\varrho_{0}}}\big)'\big|\leq\frac{K_0}{2}$,
or equivalently
$|\varrho_0'|\leq K_0\rho^{\frac32}_0$, that
 \begin{eqnarray*}
   \int_0^{T}\left\|\frac{G}{\varrho_0^{\frac\delta2}}\right\|_\infty^4
   dt&\leq
   &C\int_0^{T}\left\|\frac{G}{\varrho_0^{\frac\delta2}}\right\|_2^2
   \left\|\left(\frac{G}{\varrho_0^{\frac\delta2}}\right)_y\right\|_2^2dt\\
   &=&C\int_0^{T_0}\left\|\frac{G}{\varrho_0^{\frac\delta2}}\right\|_2^2
   \left\|\frac{G_y}{\varrho_0^{\frac\delta2}}
   -\frac\delta2\frac{\varrho_0'}{\varrho_0}\frac{G}
   {\varrho_0^{\frac\delta2} }\right\|_2^2dt\\
   &\leq&C\int_0^{T}\left\|\frac{G}{\varrho_0^{\frac\delta2}}\right\|_2^2
   \left(\left\|\frac{G_y}{\varrho_0^{\frac\delta2}}\right\|_2^2
   +\delta^2K_0^2\left\|\frac{G}
   {\varrho_0^{\frac{\delta-1}{2}}}\right\|_2^2\right)dt\\
   &\leq&C\int_0^{T}\left\|\frac{G}{\varrho_0^{\frac\delta2}}\right\|_2^2
   \left(\left\|\frac{G_y}{\varrho_0^{\frac{\delta+1}{2}}}\right\|_2^2
   +\left\|\frac{G}
   {\varrho_0^{\frac{\delta}{2}}}\right\|_2^2\right)dt<\infty
 \end{eqnarray*}
 and, thus, $\frac{G}{\varrho_0^{\frac\delta2}}\in L^4(0,T; L^\infty)$.

Next, we show the regularity that $v\in L^\infty(0,T; H^1)$, under the assumption that in this case
$\delta\geq1$ and $v_0\in H^1$. Noticing that
$\frac{G_y}{\varrho_0}\in L^2(0,T; L^2)$
and $\varrho_0v_t=G_y$, it is straightforward
that $v_t\in L^2(0,T; L^2)$ and,
consequently, recalling $v_0\in L^2$,
one obtains $v\in L^\infty(0,T; L^2)$, and further
$v\in L^\infty(0,T; H^1).$

We verify the regularity $\vartheta\in L^\infty(0,T; H^1)$, under the assumption that $\delta\geq1$ and $\frac{\pi_0}{\varrho_0}\in H^1$, as follows. In this case, one has 
$$
\frac{G}{\sqrt{\varrho_0}}\in L^\infty(0,T; L^2)\cap L^4(0,T; L^\infty),\quad
\frac{G_y}{\varrho_0}\in L^2(0,T; L^2)
$$
and, by the assumption $\big|\big(\tfrac{1}{\sqrt{\varrho_{0}}}\big)'\big|\leq\frac{K_0}{2}$, or equivalently $|\varrho_0'|\leq K_0\varrho_0^{\frac32}$, one can verify that $\frac{\pi_0'}{\varrho_0}\in L^2$.

Recalling that
\begin{equation}
  \pi_t+\frac1\mu\left(\pi+\frac{2-\gamma}{2} G \right)^2= \frac{\gamma^2}{4\mu} G ^2, \label{PI.1}
\end{equation}
one obtains
\begin{equation*}
\sup_{0\leq t\leq T}\left\|\frac{\pi}{\varrho_0}\right\|_2 \leq  \left\|\frac{\pi_0}{\varrho_0}\right\|_2+
\frac{\gamma^2}{4\mu}\int_0^{T}\left\| \frac{G}{\sqrt{\varrho_0}}\right\|_\infty\left\| \frac{G}{\sqrt{\varrho_0}}\right\|_2dt<\infty
 \label{PI.2}
\end{equation*}
and, thus, $\frac{\pi}{\varrho_0}\in L^\infty(0,T; L^2)$. Therefore, recalling $J\in L^\infty(0,T; L^\infty)$, we have
$$
\vartheta:=\frac{\pi}{R\varrho}=\frac{J\pi}{R\varrho_0}
\in L^\infty(0,T;L^2),\quad\mbox{where }\varrho:=\frac{\varrho_0}{J}.
$$

Differentiating (\ref{PI.1}) in $y$, multiplying
the resultant by $\frac{\pi_y}{\varrho_0^2}$, and integrating over $\mathbb
R$, one gets from the H\"older and Cauchy inequalities, and (\ref{pi.3}) that
\begin{eqnarray*}
\frac{d}{dt}\left\|\frac{\pi_y}{ {\varrho_0}} \right\|_2^2&\leq&C(\gamma,\mu)(\|\pi\|_\infty+\|G\|_\infty)
\left(\left\|\frac{\pi_y}{ {\varrho_0}}\right\|_2+\left\|\frac{G_y} { {\varrho_0}}\right\|_2\right) \left\|\frac{\pi_y}{ {\varrho_0}}\right\|_2\\
&\leq&C(\gamma,\mu)\left(\left\|\frac{G_y} { {\varrho_0}}\right\|_2^2+
(1+\|\pi\|_\infty^2+\|G\|_\infty^2) \left\|\frac{\pi_y}{ {\varrho_0}}\right\|_2^2\right),
\end{eqnarray*}
from which, by the Gronwall inequality, we have
$\displaystyle\sup_{0\leq t\leq T}\big\|\tfrac{\pi_y}{\sqrt{\varrho_0}}\big\|_2^2<\infty$ and, thus,
$\frac{\pi_y}{\varrho_0}\in L^\infty(0,T;L^2)$.

Rewrite equation (\ref{lcnsJ}) in terms of $G$ as $
J_t=\frac{J}{\pi}(G+\pi),$
from which, one obtains $
  J(y,t)=e^{\frac1\mu\int_0^t(G+\pi)d\tau}J_0(y)$
and, thus,
\begin{equation*}
  J_y=\left(\frac1\mu\int_0^t(G_y+\pi_y) dJ_0+J_0'\right)\tau\exp\left\{\frac1\mu\int_0^t(G+\pi)d\tau\right\}.
\end{equation*}
Hence,  
\begin{equation*}
  \sup_{0\leq t\leq T}\left\|\frac{J_y}{ {\varrho_0}}\right\|_2
  \leq \left(\frac{\bar J}{\mu} \int_0^{T_*^1}\left\|\left(\frac{G_y}{ {\varrho_0}} ,\frac{\pi_y}{ {\varrho_0}}\right) \right\|_2 dt+\left\|
  \frac{J_0'}{{\varrho_0}}\right\|_2\right) e^{\frac1\mu\int_0^{T}\|(G,\pi)\|_\infty
  dt}<\infty,
\end{equation*}
that is $\frac{J_y}{\varrho_0}\in L^\infty(0,T; L^2)$.

Thanks to the regularities $\left(\frac{\pi}{\varrho_0},\frac{\pi_y}{\varrho_0},\frac{J_y} {\varrho_0}\right)\in L^\infty(0,T;L^2)$,
noticing that $(J,\pi)\in L^\infty(0,T;L^\infty)$,
and recalling
the assumption $\big|\big(\tfrac{1}{\sqrt{\varrho_{0}}}\big)'\big|\leq\frac{K_0}{2}$, or equivalently $|\varrho_0'|\leq K_0\varrho_0^{\frac32}$, we have
$$
\vartheta_y=\frac1R\left(\frac{J\pi_y}{\varrho_0}+\frac{J_y\pi}{\varrho_0} -\frac{\varrho_0'J\pi}{\varrho_0^2}\right)\in L^\infty(0,T; L^2).
$$

It remains to prove that $s\in L^\infty(0,T; L^\infty)$, under the assumption that $s_0\in L^\infty$ and $\delta\geq\gamma$.
To this end, by the definition of $s$,
it suffice to verify that $\frac{\pi}{\varrho^\gamma}=\frac{J^\gamma\pi}{\varrho_0^\gamma}$ has
uniform positive lower and upper bounds on $\mathbb R\times(0,T)$.
Since $\delta\geq\gamma$,
it follows from (\ref{ADR1}) that $\frac{G}{\varrho_0^{\frac\gamma2}}\in L^4(0,T; L^\infty)$. To show the boundedness from the
above of $\frac{J^\gamma\pi}{\varrho_0^\gamma}$, recalling that
$J\in[\frac34\underline J,\frac54\bar J]$ on $\mathbb R\times[0,T]$, we need only to
verify that of $\frac{\pi}{\varrho_0^\gamma}$.  
(\ref{PI.1}) implies that 
\begin{eqnarray*}
  \sup_{0\leq t\leq T}\left\|\frac{\pi}{\varrho_0^\gamma}
  \right\|_\infty\leq \left\|\frac{\pi_0}{\varrho_0^\gamma}
  \right\|_\infty+\frac{\gamma^2}{4\mu}\int_0^T\left\|\frac{G}
  {\varrho_0^{\frac\gamma2}}\right\|_\infty^2dt<\infty
\end{eqnarray*}
and, thus, $\frac{\pi}{\varrho^\gamma}$ has a uniform upper bound
on $\mathbb R\times(0,T)$.
Concerning the uniform positive lower bound, by (\ref{lcnsJ}) and (\ref{lcnspi}), one has
\begin{equation*}
  (J^\gamma\pi)_t =(\gamma-1)\mu J^{\gamma-2}(v_y)^2\geq0
\end{equation*}
and, thus, $J^\gamma(y,t)\pi(y,t)\geq J_0^\gamma(y)\pi_0(y)=\pi_0(y)$, which leads to
$$
\inf_{y\in\mathbb R, t\in[0,T]}\frac{J^\gamma(y,t)\pi(y,t)}{\varrho_0^\gamma(y)}\geq\inf_{y\in\mathbb R}\frac{\pi_0(y)}{\varrho_0^\gamma(y)}>0.
$$
Hence, $\frac{\pi}{\varrho^\gamma}$ has a uniform
positive lower bound on $\mathbb R\times(0,T)$.
\end{proof}

\section{Global existence in the presence of far field vacuum}
\label{secglo}
This section is devoted to proving the global existence of strong solutions to system (\ref{lcnsJ})--(\ref{lcnspi}), subject to (\ref{IC}), which proves Theorem \ref{GLOBAL}. Throughout this section, it is always set $J_0\equiv1$.

As preparations, several a priori estimates are stated in the next
propositions. We start with 
the basic energy identity of a strong solution to system (\ref{lcnsJ})--(\ref{lcnspi}), subject to (\ref{IC}).

\begin{proposition}
\label{basic}
Given a positive time $T$, and let $(J,v,\pi)$ be a strong solution to system (\ref{lcnsJ})--(\ref{lcnspi}), subject to (\ref{IC}), on $\mathbb R\times(0,T)$, with $(\varrho_0, v_0, \pi_0)$ satisfying (H1), (H2), and (H4). Then, it holds that 
\begin{equation*}
  \int_{\mathbb R}\left(\frac{\varrho_0v^2}{2}+\frac{J\pi}{\gamma-1}\right) dy=\int_{\mathbb R}\left(\frac{\varrho_0v_0^2}{2}+\frac{\pi_0}{\gamma-1}\right) dy,\quad t\in[0,T].
\end{equation*}
\end{proposition}

\begin{proof}
Take a nonnegative function $\eta\in C_c^\infty((-2,2))$ such that $\eta\equiv1$ on $[-1,1]$ and $0\leq\eta\leq1$ on $(-2,2)$. For each $r\in(0,\infty)$, we define a
function $\eta_r$ as
$\eta_r(\cdot)=\eta(\frac{\cdot}{r})$. Multiplying (\ref{lcnsv}) and (\ref{lcnspi}), respectively, by $v\eta_r^2$ and $J\eta_r^2$, and integrating the resultants over $\mathbb R$, one gets from integration by parts that
\begin{equation}
  \label{K.1}
  \frac12\frac{d}{dt}\int_{\mathbb R}\varrho_0v^2\eta_r^2 dy+\mu\int_{\mathbb R} \frac{v_y^2\eta_r^2}{J}dy =\int_{\mathbb R}\pi v_y\eta_r^2 dy+2\int_\mathbb R\left(\pi v-\mu\frac{v_yv}{J}\right)\eta_r\eta_r'dy.
\end{equation}
and
\begin{equation}
  \label{K.2}
  \int_{\mathbb R}J\pi_t\eta_r^2dy+\gamma\int_{\mathbb R}\pi v_y\eta_r^2 dy=(\gamma-1)\mu\int_{\mathbb R}\frac{v_y^2\eta_r^2}{J}dy.
\end{equation}
(\ref{lcnsJ}) implies that 
\begin{equation*}
  \int_{\mathbb R}J\pi_t\eta_r^2dy=\frac{d}{dt}\int_{\mathbb R}J\pi\eta_r^2 dy-\int_{\mathbb R}J_t\pi\eta_r^2 dy=\frac{d}{dt}\int_{\mathbb R}J\pi\eta_r^2 dy-\int_{\mathbb R}v_y\pi\eta_r^2 dy,
\end{equation*}
which, plugged in to (\ref{K.2}) yields,
\begin{equation*}
  \frac{d}{dt}\int_{\mathbb R}J\pi\eta_r^2 dy+(\gamma-1)\int_{\mathbb R}v_y\pi\eta_r^2 dy =(\gamma-1)\mu\int_{\mathbb R}\frac{v_y^2\eta_r^2}{J}dy.
\end{equation*}
This, together with (\ref{K.1}), yields
\begin{equation*}
  \frac{d}{dt}\int_{\mathbb R}\left(\frac{\varrho_0v^2}{2}+\frac{J\pi}{\gamma-1}\right) \eta_r^2dy=2\int_\mathbb R\left(\pi-\mu\frac{v_y}{J}\right)v\eta_r\eta_r'dy,
\end{equation*}
which implies 
\begin{eqnarray}
  \int_{\mathbb R}\left(\frac{\varrho_0v^2}{2}+\frac{J\pi}{\gamma-1}\right) \eta_r^2dy&=&2\int_0^t\int_\mathbb R\left(\pi-\mu\frac{v_y}{J}\right)v\eta_r\eta_r'dyd\tau\nonumber\\
  &&+\int_{\mathbb R}\left(\frac{\varrho_0v_0^2}{2}+\frac{\pi_0}{\gamma-1}\right) \eta_r^2dy,\label{KK1}
\end{eqnarray}
for any $t\in[0,T]$.
Noticing that $|\eta_r'|\leq \frac{\|\eta'\|_\infty}{r}$, by the assumption $\varrho_0(y)\geq\frac{A_0}{(1+|y|)^2}$, one deduces
$$
|\eta_r'(y)|\leq \frac{\|\eta'\|_\infty}{r}\leq\frac{2\|\eta'\|_\infty}{|y|}\leq\frac{4\|\eta'\|_\infty} {|y|+1}\leq \frac{4\|\eta'\|_\infty}{\sqrt{A_0}}\sqrt{\varrho_0(y)},\quad\forall ~1\leq r<|y|<2r,
$$
from which, noticing that $supp\,\eta_r'\subseteq(-2r,r)\cup(r,2r)$, we
obtains
\begin{equation}
  \label{eta}
  |\eta_r'(y)|\leq  {M_0}{\sqrt{\varrho_0(y)}},\quad\forall y\in\mathbb R,\quad\mbox{ where } M_0=\frac{4\|\eta'\|_\infty}{\sqrt{A_0}},
\end{equation}
for any $r\geq1$.
Therefore, denoting
$\delta_T=\inf_{y\in\mathbb R,t\in[0,T]}J(y,t)$, and noticing
that $\sqrt{\varrho_0}v\pi$, $\sqrt{\varrho_0}vv_y\in L^1(\mathbb R\times(0,T))$, we deduce
\begin{eqnarray*}
\left|\int_0^t\int_\mathbb R\left(\pi-\mu\frac{v_y}{J}\right)v\eta_r\eta_r'dyd\tau\right|
=\left|\int_0^t\int_{r<|y|<2r} \left(\pi-\mu\frac{v_y}{J}\right)v\eta_r\eta_r'dyd\tau\right|\\
\leq M_0\int_0^t\int_{r<|y|<2r}\left(|\pi|+\frac{\mu}{\delta_T}
  |v_y|\right)|\sqrt{\varrho_0} v|dyd\tau\rightarrow0,\quad\mbox{as }r\rightarrow\infty,
\end{eqnarray*}
for any $t\in[0,T]$. Thanks to this, and taking $r\rightarrow\infty$ in (\ref{KK1}), the conclusion follows.
\end{proof}

The next proposition yields the uniform positive lower bound of $J$.

\begin{proposition}
\label{lbdJ}
Under the same assumptions as in Proposition \ref{basic}, it holds that 
\begin{equation*}
  \inf_{y\in\mathbb R}J(y,t)\geq e^{-\frac{2\sqrt{2}}{\mu}\sqrt{\mathcal E_0}\|\varrho_0\|_1},\quad t\in[0,T],
\end{equation*}
where
$\mathcal E_0:=\int_{\mathbb R}\left(\frac{\varrho_0v_0^2}{2}+\frac{\pi_0}{\gamma-1}\right) dy.$
\end{proposition}

\begin{proof}
Due to (\ref{lcnsJ}), one rewrite 
(\ref{lcnsv}) as
\begin{equation*}
  \varrho_0v_t-\mu(\log J)_{yt}+\pi_y=0.
\end{equation*}
Thus, 
\begin{equation*}
  \varrho_0v-\mu(\log J)_y+\int_0^t\pi_yd\tau=\varrho_0 v_0. 
\end{equation*}
Integrating the above identity over the interval $(z_0, z)$ yields
\begin{equation}
  \label{K.3}
  J(z,t)=J(z_0,t)\exp\left\{\frac1\mu \int_{z_0}^z\varrho_0(v-v_0) dz\right\}\exp\left\{\frac1\mu \int_0^t(\pi-\pi(z_0,\tau))d\tau\right\},
\end{equation}
for any $z,z_0\in\mathbb R$ and $t\in(0,T_\infty)$. By Proposition \ref{basic}, it follows from the H\"older inequality that
$$
\left|\int_{z_0}^z\varrho_0(v-v_0)dy\right|\leq\|\varrho_0\|_1(\|\sqrt{\varrho_0} v\|_2+\|\sqrt{\varrho_0}v_0\|_2)\leq 2\sqrt{2\mathcal E_0}\|\varrho_0\|_1.
$$
It then follows from this, $\pi\geq0$, and (\ref{K.3}) that
\begin{equation}
\label{K.4}
  J(z,t)\geq \exp\left\{-\frac{2\sqrt{2}}{\mu}\sqrt{\mathcal E_0}\|\varrho_0\|_1\right\}\exp\left\{-\frac1\mu \int_0^t \pi(z_0,\tau) d\tau\right\}J(z_0,t),
\end{equation}
for any $z,z_0\in\mathbb R$ and $t\in(0,T_\infty)$.

Since $v_y\in L^2(0,T; H^1)$, it is clear
that, for any fixed $t\in[0,T]$, $\int_0^tv_yd\tau\in H^1$. Thanks to this, and using the fact that $f(y)\rightarrow0$, as $y\rightarrow\infty$, for any $f\in H^1$, one obtains from equation (\ref{lcnsJ}) that
\begin{equation}
\label{K.5}
J(y,t)=1+\int_0^tv_y(y,\tau) d\tau\rightarrow1,\quad\mbox{as }y\rightarrow\infty,
\end{equation}
for any $t\in[0,T]$. By the Sobolev embedding inequality, it follows
from $v_y\in L^2(0,T; H^1)$ that $v_y\in L^2(0,T; L^\infty)$; therefore, recalling $v_y\in L^\infty(0,T; L^2)$, it follows from the H\"older inequality that
$$
(v_y)^2\in L^2(0,T; L^2), \quad v_yv_{yy}\in L^1(0,T; L^2),
$$
which imply $\int_0^t(v_y)^2d\tau\in H^1$, for any fixed $t\in[0,T]$.
Thanks to this, and using again the fact that $f(y)\rightarrow0$, as $y\rightarrow\infty$, for any $f\in H^1$, one has 
$$
\lim_{y\rightarrow\infty}\int_0^t|v_y|(y,\tau) d\tau\leq\sqrt t\left(\lim_{y\rightarrow\infty}\int_0^t(v_y)^2(y,\tau) d\tau\right)^{\frac12} =0,
$$
for any $t\in[0,T]$.
Therefore, denoting
$\delta_T=\inf_{y\in\mathbb R, t\in(0,T)}J(y,t)$, and solving equation (\ref{lcnspi}), we deduce
\begin{eqnarray}
  \pi(y,t)&=&\exp\left\{-\gamma\int_0^t\frac{v_y}{J}d\tau
  \right\}\left(\pi_0(y)+(\gamma-1)\mu \int_0^te^{\gamma\int_0^\tau\frac{v_y}{J}d\tau'}\left(\frac{v_y}{J}\right)^2 d\tau\right)\nonumber\\
  &\leq&e^{\frac{\gamma}{\delta_T}\int_0^t|v_y|d\tau} \left(\pi_0(y)+\frac{(\gamma-1)\mu}{\delta_T^2} \int_0^tv_y^2 d\tau\right)
  \rightarrow0,\quad\mbox{as }y\rightarrow\infty,\label{K.6}
\end{eqnarray}
for any $t\in[0,T]$.

Due to (\ref{K.5}) and (\ref{K.6}), by taking $z_0\rightarrow\infty$, one obtains from (\ref{K.4}) that
\begin{equation*}
  \inf_{y\in\mathbb R}J(y,t)\geq e^{-\frac{2\sqrt{2}}{\mu}\sqrt{\mathcal E_0}\|\varrho_0\|_1},
\end{equation*}
for any $t\in[0,T]$, proving the conclusion.
\end{proof}

The next proposition gives the conservation of momentum.

\begin{proposition}
\label{momentum}
  Under the same assumptions as in Proposition \ref{basic}, it holds that 
  $$
  \int_\mathbb R\varrho_0(y)v(y,t)dy=\int_\mathbb R\varrho_0(y)v_0(y)dy,\quad\forall t\in[0,T].
  $$
\end{proposition}

\begin{proof}
Let $\eta$ and $\eta_r$ be the same functions as in the proof of Proposition \ref{basic}. Multiplying equation (\ref{lcnsv}) by $\eta_r$, and integrating
the resultant over $\mathbb R$, one gets by integration by parts that
\begin{equation*}
  \frac{d}{dt}\int_\mathbb R\varrho_0v\eta_rdy=\int_\mathbb R\left(\pi-\mu\frac{v_y}{J}\right)\eta_r'dy,
\end{equation*}
which implies that 
\begin{equation}\label{EQN1}
  \int_\mathbb R\varrho_0v\eta_rdy=\int_\mathbb R\varrho_0v_0\eta_r dy
  +\int_0^t\int_\mathbb R\left(\pi-\mu\frac{v_y}{J}\right)\eta_r'dyd\tau.
\end{equation}
We claim that
\begin{equation}
  \lim_{r\rightarrow\infty}\int_0^t\int_\mathbb R\left(\pi-\mu\frac{v_y}{J}\right)\eta_r'dyd\tau=0, \label{EQN2}
\end{equation}
thanks to which, noticing that $\varrho_0v\in L^1(\mathbb R)$ and
$\varrho_0v_0\in L^1(\mathbb R)$, one obtains the desired conclusion by taking the
limit $r\rightarrow\infty$ in (\ref{EQN1}).

It remains to verify
(\ref{EQN2}).
To this end, recalling
$supp\,\eta_r\subseteq(-2r,-r)\cup(r,2r)$ and
(\ref{eta}), by Proposition \ref{lbdJ}, and
noticing that $\sqrt{\varrho_0}\pi\in L^1(\mathbb R\times(0,T))$ and
$\sqrt{\varrho_0}v_y\in L^1(\mathbb R\times(0,T))$, one can get 
\begin{eqnarray*}
\left|\int_0^t\int_\mathbb R \pi \eta_r'dyd\tau\right|
&=&\left|\int_0^t\int_{r<|y|<2r}\pi\eta_r'dyd\tau\right|\\
  &\leq& M_0\int_0^t\int_{r<|y|<2r}\sqrt{\varrho_0}\pi dyd\tau\rightarrow0,\quad\mbox{as }r\rightarrow\infty,
\end{eqnarray*}
and
\begin{eqnarray*}
\left|\int_0^t\int_{\mathbb R}\frac{v_y}{J}\eta_r'dyd\tau\right|
&=&\left|\int_0^t\int_{r<|y|<2r}\frac{v_y}{J}\eta_r'dyd\tau\right|\\
  &\leq&\frac{M_0}{c_0}\int_0^t\int_{r<|y|<2r}\sqrt{\varrho_0}|v_y|dyd\tau
  \rightarrow0,\quad\mbox{as }r\rightarrow\infty,
\end{eqnarray*}
where $c_0=e^{-\frac{2\sqrt{2}}{\mu}\sqrt{\mathcal
E_0}\|\varrho_0\|_1}$. Therefore, (\ref{EQN2}) holds and, consequently,
the conclusion follows.
\end{proof}

The following proposition on the global in time a priori estimates on
the effective viscous flux $G$ is the key for proving
the global existence
of strong solutions.

\begin{proposition}
\label{estG}
Under the same assumptions as in Proposition \ref{basic}, it holds that 
\begin{eqnarray*}
&\displaystyle\sup_{0\leq t\leq T} \|G\|_2^2 +\int_0^{T} \left\|
\frac{G_y} {\sqrt{\varrho_0}}\right\|_2^2 dt+
\left(\int_0^{T} \|G\|_\infty^4dt\right)^{\frac12}\leq Ce^{CT}\|G_0\|_2^2 ,
\end{eqnarray*}
for a positive constant $C$ depending only on $\gamma,\mu,\bar\varrho,\mathcal E_0$, and $\|\varrho_0\|_1$.
\end{proposition}

\begin{proof}
Let $\eta$ and $\eta_r$ be the same functions as in the proof of Proposition \ref{basic}.
Note that $G$ satisfies (\ref{eqG}). Multiplying (\ref{eqG}) by $JG\eta^2_r$, and integrating the resultant over
$\mathbb R$, one gets by integration by parts that
\begin{equation}\label{KG.1}
  \int_{\mathbb R}JGG_t\eta^2_rdy+\mu\int_{\mathbb
  R}\frac{(G_y)^2\eta^2_r}{\varrho_0}dy=-\gamma\int_{\mathbb R}
  v_y G^2\eta^2_rdy-2\mu\int_\mathbb R\frac{G_y}{\varrho_0}G\eta_r\eta_r'dy.
\end{equation}
On the other hand, (\ref{lcnsJ}) implies that 
\begin{equation*}
  \int_{\mathbb R}JGG_t\eta^2_rdy =  \frac12\left(\frac{d}{dt}\int_{\mathbb R}JG^2\eta^2_rdy- \int_{\mathbb R}v_y
  G^2\eta^2_rdy\right).
\end{equation*}
Plugging the above into (\ref{KG.1}) and
integrating by parts show that
\begin{eqnarray*}
   \frac12\frac{d}{dt}\int_{\mathbb R} JG^2\eta^2_rdy+  \mu\int_{\mathbb
  R}\frac{(G_y)^2\eta^2_r}{\varrho_0}dy =\left(\frac12-\gamma\right)\int_{\mathbb R} v_yG^2\eta_r^2 dy
   -2\mu\int_\mathbb R\frac{GG_y}{\varrho_0}\eta_r\eta_r'dy\nonumber\\
   =(2\gamma-1)\int_\mathbb R vGG_y\eta_r^2dy+\int_{\mathbb R}\left[(2\gamma-1)vG^2-2\mu\frac{GG_y}{\varrho_0}\right]\eta_r\eta_r' dy,
\end{eqnarray*}
which implies that 
\begin{eqnarray}
  &&\int_{\mathbb R} JG^2\eta^2_rdy+  2\mu\int_0^t\int_{\mathbb
  R}\frac{(G_y)^2\eta^2_r}{\varrho_0}dy d\tau\nonumber \\
  &=&2\int_0^t\int_{\mathbb R}\left((2\gamma-1)vG^2-2\mu\frac{G_y}{\varrho_0}G\right)\eta_r\eta_r' dyd\tau\nonumber\\
  &&+(4\gamma-4)\int_0^t\int_\mathbb R vGG_y\eta_r^2dyd\tau+\int_\mathbb RG_0^2\eta_r^2dy,\label{KG.2}
\end{eqnarray}
for any $t\in[0,T]$.

We are going to show that
\begin{eqnarray}
&&\lim_{r\rightarrow\infty}\int_0^t\int_\mathbb R vGG_y\eta_r^2dyd\tau=\int_0^t\int_\mathbb R vGG_ydyd\tau,\label{KG.3}\\
&&\lim_{r\rightarrow\infty}\int_0^t\int_{\mathbb R} vG^2 \eta_r\eta_r' dyd\tau=\lim_{r\rightarrow\infty}\int_0^t\int_{\mathbb R} \frac{G_y}{\varrho_0}G \eta_r\eta_r' dyd\tau=0, \label{KG.4}
\end{eqnarray}
for any $t\in[0,T]$. To verify (\ref{KG.3}), it suffices to show $vGG_y\in L^1(\mathbb R\times(0,T))$. By the regularities of $(J, v,
\pi)$, one can check that
$$
G\in L^\infty(0,T; L^2),\quad
\frac{G_y}{\sqrt{\varrho_0}}\in L^2(0,T; L^2),
$$
and further, by the Sobolev embedding, that $G\in L^2(0,T; L^\infty)$. As a result, by the H\"older
inequality, we have
$vGG_y=\sqrt{\varrho_0}vG\frac{G_y}{\sqrt{\varrho_0}}\in L^1(0,T; L^1)$ and, thus, (\ref{KG.3}) holds.

Observing that $\sqrt{\varrho_0} vG^2\in L^1(0,T; L^1)$ and $\frac{G_yG}{\sqrt{\varrho_0}}\in L^1(0,T; L^1)$, and recalling
$supp\,\eta_r\subset(-2r,r)\cup(r,2r)$ and (\ref{eta}), we have
\begin{eqnarray*}
  \left|\int_0^t\int_{\mathbb R}vG^2\eta_r\eta_r'dyd\tau\right|&=&
  \left|\int_0^t\int_{r<|y|<2r}vG^2\eta_r\eta_r'dyd\tau\right|\\
  &\leq & M_0\int_0^t\int_{r<|y|<2r}\sqrt{\varrho_0}|v|G^2dyd\tau\rightarrow0, \quad\mbox{as }r\rightarrow\infty,
\end{eqnarray*}
and
\begin{eqnarray*}
 \left|\int_0^t\int_{\mathbb R} \frac{G_y}{\varrho_0}G \eta_r\eta_r' dyd\tau\right|&=&\left|\int_0^t\int_{r<|y|<2r} \frac{G_y}{\varrho_0}G \eta_r\eta_r' dyd\tau\right|\\
 &\leq& M_0\int_0^t\int_{r<|y|<2r}\frac{|G_y||G|}{\sqrt{\varrho_0}}dyd\tau
 \rightarrow0,\quad\mbox{as }r\rightarrow\infty,
\end{eqnarray*}
for any $t\in[0,T]$. Therefore, (\ref{KG.4}) holds.

Due to (\ref{KG.3}) and (\ref{KG.4}), by taking $r\rightarrow\infty$,
one obtains from (\ref{KG.2}) that
\begin{eqnarray*}
 \int_{\mathbb R} JG^2 dy+  2\mu\int_0^t\int_{\mathbb
  R}\frac{(G_y)^2}{\varrho_0}dy d\tau=(4\gamma-2)\int_0^t\int_\mathbb R vGG_y dyd\tau+\int_\mathbb RG_0^2 dy,
\end{eqnarray*}
which and Proposition \ref{lbdJ} give 
\begin{eqnarray}
 \|G\|_2^2(t)+\int_0^t\left\|\frac{G_y}{\sqrt{\varrho_0}}\right\|_2^2 d\tau\leq C\left(\int_0^t\int_\mathbb R |v||G||G_y| dyd\tau+\|G_0\|_2^2 \right),\label{KG.5}
\end{eqnarray}
for any $t\in[0,T]$, where $C$ is a positive constant depending only
on $\gamma,\mu,\mathcal E_0,$ and $\|\varrho_0\|_1$. It follows from the
the H\"older, Yong, and Gagliardo-Nirenberg inequalities, and Proposition \ref{basic} that 
\begin{eqnarray*}
  &&C\int_{\mathbb R}|v||G||G_y|dy=C\int_{\mathbb R}\sqrt{\varrho_0}|v||G|\frac{|G_y}{\sqrt{\varrho_0}}dy \\
  &\leq& C\|\sqrt{\varrho_0}v\|_2\|G\|_\infty\left\|\frac{G_y}{\sqrt{\varrho_0}}
  \right\|_2\leq C\|\sqrt{\varrho_0}v\|_2\|G\|_2\|G_y\|_2^{\frac12}  \left\|\frac{G_y}{\sqrt{\varrho_0}}
  \right\|_2\\
  &\leq& C\|G\|_2^{\frac12}\left\|\frac{G_y}{\sqrt{\varrho_0}}
  \right\|_2^{\frac32}\leq\frac12\left\|\frac{G_y}{\sqrt{\varrho_0}}
  \right\|_2^2+C\|G\|_2^2,
\end{eqnarray*}
for a positive constant $C$ depending only on $\gamma,\mu,\bar\varrho,\mathcal E_0$, and $\|\varrho_0\|_1$. Plugging the above estimate into (\ref{KG.5}) yields
$$
\|G\|_2^2(t)+\int_0^t\left\|\frac{G_y}{\sqrt{\varrho_0}}\right\|_2^2 d\tau\leq C\left(\int_0^t\|G\|_2^2d\tau+\|G_0\|_2^2 \right),
$$
for any $t\in[0,T]$, for a positive constant $C$ depending only on $\gamma,\mu,\bar\varrho,\mathcal E_0$, and $\|\varrho_0\|_1$. Then 
the Gronwall inequality shows that 
$$
\sup_{0\leq t\leq T}\|G\|_2^2+\int_0^T\left\|\frac{G_y}{\sqrt{\varrho_0}}\right\|_2^2 dt
\leq Ce^{CT}\|G_0\|_2^2
$$
and, consequently, by
the Gagliardo-Nirenberg inequality, one has 
$$
\int_0^T\|G\|_\infty^4dt\leq C\int_0^T\|G\|_2^2\|G_y\|_2^2dt\leq Ce^{CT}\|G_0\|_2^4.
$$
for a positive constant $C$ depending only on $\gamma,\mu,\bar\varrho,\mathcal E_0$, and $\|\varrho_0\|_1$.
\end{proof}

Based on Proposition \ref{estG}, similar to Proposition \ref{proppi},
one can obtain the other a priori estimates stated in the next proposition.

\begin{proposition}
  \label{estother}
Under the same assumptions as in Proposition \ref{basic}, it holds that 
\begin{equation*}
\sup_{0\leq t\leq T}\left\|\left(J-1,\frac{J_y}{\sqrt{\varrho_0}},J_t,v_y, \pi,\frac{\pi_y}{\sqrt{\varrho_0}}\right)
\right\|_2\leq C,
\end{equation*}
and
\begin{equation*}
\int_0^T\left(
\|\pi_t\|_2^4+\left\|\left(\frac{v_{yy}}{\sqrt{\varrho_0}}, \sqrt{\varrho_0}v_t\right)\right\|_2^2\right)dt\leq C,
\end{equation*}
for a positive constant $C$ depending only on
$\gamma,\mu,\bar\varrho,\mathcal E_0$, $\|\varrho_0\|_1$, $\|\pi_0\|_2$,
$\big\|\frac{\pi_0'}{\sqrt{\varrho_0}}\big\|_2$, and $T$, and the constant $C$, viewing as a function of $T$,
is continuously increasing with respect to $T\in[0,\infty)$.
\end{proposition}

\begin{proof}
The proof is almost the same as that for Proposition \ref{proppi},
and the only difference is that the role played by Proposition
\ref{propGphi} there is now played by Proposition \ref{estG}, and
we have to estimate the upper bound of $J$ on $[0,T]$, which is
automatically guaranteed there by the assumption, before proving the
estimates on $v_y$ and $\frac{v_{yy}}{\sqrt{\varrho_0}}$. Therefore,
we only sketch the proof here.

Repeating the arguments as those for
(\ref{pi.2}), (\ref{pi.3-1}), (\ref{pi.4}), (\ref{pi.5}), and (\ref{pi.6}), but applying Proposition \ref{estG}, instead of applying Proposition \ref{propGphi} there, one obtains
\begin{equation}
\sup_{0\leq t\leq T}\left\|\left(\frac{J_y}{\sqrt{\varrho_0}}, \pi,\frac{\pi_y}{\sqrt{\varrho_0}}\right)
\right\|_2+\int_0^T(\|\pi_t\|_2^4+\|\sqrt{\varrho_0}v_t\|_2^2)  dt\leq C,\label{NE1}
\end{equation}
here and throughout the proof of this proposition, $C$ is a positive constant depending only on
$\gamma,\mu,\bar\varrho,\mathcal E_0$, $\|\varrho_0\|_1$, $\|\pi_0\|_2$,
$\big\|\frac{\pi_0'}{\sqrt{\varrho_0}}\big\|_2$, and $T$, and it 
is continuously increasing with respect to $T\in[0,\infty)$. Thanks to
this estimate and the Sobolev embedding inequality, we have
$\displaystyle\sup_{0\leq t\leq T}\|\pi\|_\infty\leq C$.
Noticing that $J(y,t)=\exp\left\{\frac1\mu\int_0^t(G+\pi)d\tau\right\}$,
it follows from Proposition \ref{estG} and (\ref{NE1})
that $\displaystyle\sup_{0\leq t\leq T}\|J\|_\infty\leq C$.
Thanks to (\ref{NE1}) and $\displaystyle\sup_{0\leq t\leq T}(\|\pi\|_\infty+\|J\|_\infty)\leq C$, repeating the arguments as those
for (\ref{J.1-1}), (\ref{pi.5-1}), and (\ref{pi.7}), but applying Proposition \ref{estG}, instead of applying Proposition \ref{propGphi} there, 
one can get 
\begin{equation*}
  \sup_{0\leq t\leq T}\|(J-1,J_t,v_y)\|_2^2+\int_0^T\left\|\frac{v_{yy}}{
  \sqrt{\varrho_0}}\right\|_2^2dt\leq C.
\end{equation*}
Thus, Proposition \ref{estother} is proved.
\end{proof}

The next proposition will be the key to show the global in time
boundedness of the entropy.

\begin{proposition}
  \label{estGextra}
Let the assumption in Proposition \ref{basic} hold. Assume in addition that (H3) holds and $\frac{G}{{\varrho_0^{\frac\delta2}}}\in L^2(0,T;L^2)$. Then, we have the following estimate:
$$
\sup_{0\leq t\leq T}\left\|\frac{ G}{\varrho_0^{\frac\delta2}}\right\|_2^2
+\int_0^T\left\|\frac{G_y}{\varrho_0^{\frac{\delta+1}{2}}} \right\|_2^2dt
\leq C \left\|\frac{ G_0}{\varrho_0^{\frac\delta2}}\right\|_2^2,
$$
for a positive constant $C$ depending only on
$\gamma,\mu,\delta,K_0,\bar\varrho,\mathcal E_0$, $\|\varrho_0\|_1$, $\|\pi_0\|_2$,
$\big\|\frac{\pi_0'}{\sqrt{\varrho_0}}\big\|_2$, and $T$, and the constant $C$, viewing as a function of $T$,
is continuously increasing with respect to $T\in[0,\infty)$.
\end{proposition}

\begin{proof}
  Let $\eta_r$ be the same function as in the proof of Proposition \ref{basic}. Note that $G$ satisfies (\ref{eqG}). Multiplying equation (\ref{eqG}) by $\frac{JG\eta^2_r}{\varrho_0^\delta}$, and integrating the resultant over
$\mathbb R$, it follows from integration by parts that
\begin{equation}
  \int_{\mathbb R}\frac{JG\eta_r^2}{\varrho_0^\delta}G_t dy +
  \mu\int_{\mathbb R}\frac{G_y}{\varrho_0}\left(\frac{G\eta_r^2}{\varrho_0^\delta}\right)_y
  dy=-\gamma\int_{\mathbb R}v_y\frac{G^2\eta_r^2}{\varrho_0^\gamma}dy.
  \label{E1}
\end{equation}
Direct calculations yield
\begin{equation}
  \label{E2}
  \int_{\mathbb R}\frac{G_y}{\varrho_0}\left(\frac{G\eta_r^2}{\varrho_0^\delta}\right)_y
  dy=\int_{\mathbb R}\frac{G_y}{\varrho_0}\left[
  \left(\frac{G_y}{\varrho_0^\delta}-\delta\frac{\varrho_0'}{\varrho_0}
  \frac{G}{\varrho_0^\delta}\right)\eta_r^2+2\frac{G}{\varrho_0^\delta}
  \eta_r\eta_r'\right]dy.
\end{equation}
Using equation (\ref{lcnsJ}), one can get 
\begin{equation}
  \int_{\mathbb R}\frac{JG\eta_r^2}{\varrho_0^\delta}G_t dy
  =\frac12\left(\frac{d}{dt}\int_{\mathbb R}
  \frac{JG^2\eta_r^2}{\varrho_0^\delta}dy-\int_{\mathbb R} v_y\frac{G^2 \eta_r^2}{\varrho_0^\delta}dy\right).\label{E3}
\end{equation}
Plugging (\ref{E2}) and (\ref{E3}) into (\ref{E1}), one can get from
the assumption $\big|\big(\tfrac{1}{\sqrt{\varrho_{0}}}\big)'\big|\leq\frac{K_0}{2}$, or equivalently $|\varrho_0'|\leq K_0\varrho_0^{\frac32}$, and the Cauchy
inequality that
\begin{eqnarray}
  &&\frac12\frac{d}{dt}\int_{\mathbb R}
  \frac{JG^2\eta_r^2}{\varrho_0^\delta}dy+\mu\int_{\mathbb R}\frac{G_y^2\eta_r^2}{\varrho_0^{\delta+1}}dy\nonumber\\ &=&\left(\frac12-\gamma\right) \int_{\mathbb R} v_y\frac{G^2 \eta_r^2}{\varrho_0^\delta}dy+\mu\int_{\mathbb R}\left(\delta\frac{\varrho_0'}{\varrho_0}\eta_r^2
  -2\eta_r\eta_r'\right)
  \frac{GG_y}{\varrho_0^{\delta+1}}dy\nonumber\\
  &\leq&\gamma\int_{\mathbb R} |v_y|\frac{G^2 \eta_r^2} {\varrho_0^\delta}dy+\mu\int_{\mathbb R}(\delta K_0\sqrt{\varrho_0}\eta_r^2+2\eta_r|\eta_r'|)\frac{GG_y} {\varrho_0^{\delta+1}}dy\nonumber\\
  &\leq&\frac\mu2\int_{\mathbb R}\frac{G_y^2\eta_r^2}{\varrho_0^{\delta+1}}dy+C\int_{\mathbb R}\left[(|v_y|+1)\frac{G^2\eta_r^2}{\varrho_0^{\delta}}+ \frac{G^2}{\varrho_0^{\delta+1}}|\eta_r'|^2\right]dy,\label{E4}
\end{eqnarray}
for a positive constant $C$ depending only on $\gamma,\mu,\delta$, and $K_0$.
It follows from $supp\,\eta_r\subseteq(-2r,r)\cup(r,2r)$ and (\ref{eta})
that
$$
\int_{\mathbb R}\frac{G^2}{\varrho_0^{\delta+1}}|\eta_r'|^2 dy
\leq M_0^2\int_{r<|y|<2r}\frac{G^2}{\varrho_0^{\delta}}dy,\quad r\geq1.
$$
This, together with Proposition \ref{lbdJ}
and (\ref{E4}), implies that
\begin{eqnarray}
  &&\frac{d}{dt}\left\|\sqrt{\frac{J}{\varrho_0^\delta}}G\eta_r\right\|_2^2
  +\mu\left\|\frac{G_y\eta_r}{\varrho_0^{\frac{\delta+1}{2}}}\right\|_2^2
  \nonumber\\
  &\leq&C(\|v_y\|_\infty+1) \left\|\sqrt{\frac{J}{\varrho_0^\delta}}G\eta_r\right\|_2^2 +
  CM_0^2\int_{r<|y|<2r}\frac{G^2}{\varrho_0^{\delta}}dy,\label{E5}
\end{eqnarray}
for $r\geq1$, and for a positive constant $C$ depending only on
$\gamma,\mu,\delta,K_0,\mathcal E_0$, and $\|\varrho_0\|_1$. By
Proposition \ref{estother}, it follows from the Sobolev embedding
inequality that $\int_0^T\|v_y\|_\infty dt\leq C$, for a positive
constant $C$ depending only on
$\gamma,\mu,\bar\varrho,\mathcal E_0$, $\|\varrho_0\|_1$, $\|\pi_0\|_2$,
$\big\|\frac{\pi_0'}{\sqrt{\varrho_0}}\big\|_2$, and $T$, and $C$
is continuously increasing with respect to $T\in[0,\infty)$. Thanks to this, by applying the Gronwall inequality to (\ref{E5}), we obtain
$$
\sup_{0\leq t\leq T}\left\|\sqrt{\frac{J}{\varrho_0^\delta}}G\eta_r\right\|_2^2
+\int_0^T\left\|\frac{G_y\eta_r}{\varrho_0^{\frac{\delta+1}{2}}} \right\|_2^2dt
  \leq C\left( \left\|\frac{ G_0}{\varrho_0^{\frac\delta2}}\right\|_2^2 +M_0^2\int_0^T
  \int_{r<|y|<2r}\frac{G^2}{\varrho_0^{\delta}}dydt\right),
$$
for $r\geq1$, from which, noticing that
the assumption  $\frac{G}{{\varrho_0^{\frac\delta2}}}\in L^2(0,T; L^2))$ implies that the last term of the right hand side of the above inequality tends to zero, as $r\rightarrow\infty$, we obtains by taking $r\rightarrow\infty$ that
$$
\sup_{0\leq t\leq T}\left\|\sqrt{\frac{J}{\varrho_0^\delta}}G\right\|_2^2
+\int_0^T\left\|\frac{G_y}{\varrho_0^{\frac{\delta+1}{2}}} \right\|_2^2dt
\leq C \left\|\frac{ G_0}{\varrho_0^{\frac\delta2}}\right\|_2^2,
$$
for a positive constant $C$ depending only on
$\gamma,\mu,\delta,K_0,\bar\varrho,\mathcal E_0$, $\|\varrho_0\|_1$, $\|\pi_0\|_2$,
$\big\|\frac{\pi_0'}{\sqrt{\varrho_0}}\big\|_2$, and $T$, and $C$
is continuously increasing with respect to $T\in[0,\infty)$. This, together
with Proposition \ref{lbdJ}, yields the desired conclusion.
\end{proof}

We are now ready to prove Theorem \ref{GLOBAL}:

\begin{proof}[\textbf{Proof of Theorem \ref{GLOBAL}}]
(i) By (i) of Theorem \ref{LOCAL}, there is a strong solution
$(J,v,\pi)$ to system (\ref{lcnsJ})--(\ref{lcnspi}), subject to
(\ref{IC}), on $\mathbb R\times(0,T_0)$, for some positive time $T_0$.
Extend this local strong solution to the maximal time of existence
$T_*$. Then, $(J,v,\pi)$ is a strong solution to system (\ref{lcnsJ})--(\ref{lcnspi}), subject to
(\ref{IC}), one $\mathbb R\times(0,T)$, for any $T\in(0,T_*)$.
Denote $\mathcal E_0=\int_{\mathbb R}\left(\frac{\varrho_0v_0^2}{2}+\frac{\pi_0}{\gamma-1}\right) dy.$
By Propositions \ref{basic}--\ref{momentum}, we have
\begin{eqnarray*}
&&\left[\int_{\mathbb R}\left(\frac{\varrho_0v^2}{2}+\frac{J\pi}{\gamma-1}\right) dy\right](t)=\mathcal E_0, \\
&&\inf_{y\in\mathbb R}J(y,t)\geq \exp\left\{-\frac{2\sqrt{2}}{\mu}\sqrt{\mathcal E_0}\|\varrho_0\|_1\right\},\\
&&\left(\int_{\mathbb R}\varrho_0vdy\right)(t)=\int_\mathbb R\varrho_0v_0dy,
\end{eqnarray*}
for any $t\in[0,T_*)$.
In order to prove the conclusion, it suffices to show that $T_*=\infty$.

Assume, by contradiction, that $T_*<\infty$. By Proposition \ref{estother}, we have
\begin{eqnarray*}
&&\sup_{0\leq t\leq T}\left\|\left(J-1,\frac{J_y}{\sqrt{\varrho_0}},J_t,v_y, \pi,\frac{\pi_y}{\sqrt{\varrho_0}}\right)
\right\|_2\leq C,\\
&&\int_0^T\left(
\|\pi_t\|_2^4+\left\|\left(\frac{v_{yy}}{\sqrt{\varrho_0}}, \sqrt{\varrho_0}v_t\right)\right\|_2^2\right)dt\leq C,
\end{eqnarray*}
for any $T\in[0,T_*)$, where $C$ is a positive constant depending only on
$\gamma$, $\mu$, $\bar\varrho$, $\mathcal E_0$, $\|\varrho_0\|_1$, $\|\pi_0\|_2$,
$\big\|\frac{\pi_0'}{\sqrt{\varrho_0}}\big\|_2$, and $T$, and this
constant $C$, viewing as a function of $T$, is
continuously increasing with
respect to $T\in[0,\infty)$. Since $T_*$ is a finite positive number,
we can see
the positive constants in the above are actually independent
of $T\in(0,T_*)$. Thanks to the above estimates, we have
$$
\inf_{y\in\mathbb R}J(y,T_1)>0,\quad J(\cdot,T_1)\in L^\infty, \quad\left(\frac{\partial_yJ}{\sqrt{\varrho_0}},
\sqrt{\varrho_0}v,\partial_yv,
\pi,\frac{\partial_y\pi}{\sqrt{\varrho_0}}\right)(\cdot,T_1)\in L^2,
$$
and
$$
\frac{1}{\displaystyle\inf_{y\in\mathbb R}J(y,T_1)}+\left(\|J\|_\infty+\|G\|_2
+\|\pi\|_\infty\right)(T_1)
\leq C_0,
$$
for a positive constant $C_0$
independent of $T_1\in(0,T_*)$. Therefore, by Theorem \ref{LOCAL},
there is a positive time $t_0$, such that, starting from time
$T_*-\frac{t_0}{2}$ , one can extend the strong solution $(J,v,\pi)$
uniquely to another time $T_*-\frac{t_0}{2}+t_0=T_*+ \frac{t_0}{2}>T_*$, 
which contradicts
to the definition of $T_*$. Therefore, $T_*=\infty$ and, thus,
one obtains a unique global strong solution to system (\ref{lcnsJ})--(\ref{lcnspi}), subject to (\ref{IC}).

(ii) We only prove that (\ref{ADR1}) holds for any finite $T\in(0,\infty)$, while the validity of (\ref{ADR2}) follows from (\ref{ADR1}) and (i), by exactly the same arguments as in the proof of Theorem \ref{LOCAL}. By (ii) of Theorem \ref{LOCAL}, there is a positive time $T$, such that (\ref{ADR1}) holds. Denote by $T_\ell$ the maximal
time, such that (\ref{ADR1}) holds, for any $T\in(0,T_\ell)$. In order to verify that (\ref{ADR1}) holds for any finite $T\in(0,\infty)$, it suffices to show that $T_\ell=\infty$. Assume, by contradiction, that $T_\ell<\infty$. By Proposition \ref{estGextra}, the following
estimate holds 
$$
\sup_{0\leq t\leq T}\left\|\frac{ G}{\varrho_0^{\frac\delta2}}\right\|_2^2
+\int_0^T\left\|\frac{G_y}{\varrho_0^{\frac{\delta+1}{2}}} \right\|_2^2dt
\leq C \left\|\frac{ G_0}{\varrho_0^{\frac\delta2}}\right\|_2^2,
$$
for any $T\in(0,T_\ell)$, where $C$ is a positive constant
depending only on
$\gamma,\mu,\delta,K_0,\bar\varrho,\mathcal E_0$, $\|\varrho_0\|_1, \|\pi_0\|_2,\big\|\frac{\pi_0'}{\sqrt{\varrho_0}}\big\|_2$, and $T$, and this constant $C$, viewing as a function of $T$,
is continuously increasing with respect to $T\in[0,\infty)$. Since $T_\ell$ is a positive finite number, the above constant $C$ is actually
independent of $T\in(0,T_\ell)$. Due to this fact, one can see that
$$
\frac{G(\cdot,T_\ell)}{\varrho_0^{\frac\delta2}(\cdot)}
\Bigg|_{t=T_\ell}\in L^2.
$$
With the aid of this, taking $T_\ell$ as the initial time, by Theorem \ref{LOCAL}, one can see that (\ref{ADR1}) holds for some other time $T_\ell'>T_\ell$, which contradicts to the definition of $T_\ell$. Therefore, one must have $T_\ell=\infty$, in other words, (\ref{ADR1}) holds for any finite time $T\in(0,\infty)$.
\end{proof}

\section*{Acknowledgments}
This work was supported in part by the
Zheng Ge Ru Foundation, Hong Kong RGC Earmarked Research Grants
CUHK-14305315 and CUHK-4048/13P, NSFC/RGC Joint Research Scheme Grant
N-CUHK443/14, and a Focus Area Grant from the Chinese University of Hong
Kong. J.L. would like to thank Mathematisches Forschungsinstitut
Oberwolfach and the Technische Universitat Darmstadt for the warm
and kind hospitalities where part of this work was done during the
visit to these two places as a Simons Visiting Professor (SVP).
This work was also supported in part by
the Direct Grant for Research 2016/2017 (Project Code: 4053216) from The
Chinese University of Hong Kong.

\end{document}